\documentclass{amsart} 
\usepackage{framed,multirow}

\usepackage{latexsym}
\usepackage{array}
\usepackage{epsf}
\usepackage{amsmath,bm}
\usepackage{amsfonts}
\usepackage{fancyhdr}
\usepackage{pxfonts}
\usepackage{caption}
\usepackage{epstopdf}
\usepackage{marginnote}
\usepackage[ruled,linesnumbered]{algorithm2e}
\usepackage{empheq}
\usepackage{stmaryrd}
\usepackage{color}
\usepackage{amssymb}
\usepackage{latexsym} 
\usepackage{graphicx}
\usepackage[caption=false]{subfig} 
\usepackage{lipsum}
\usepackage{tikz}
\usetikzlibrary{arrows,shapes,chains,positioning}
\usetikzlibrary{decorations.markings}
\tikzstyle arrowstyle=[scale=1]
\tikzstyle directed=[postaction={decorate,decoration={markings,mark=at position .65 with {\arrow[arrowstyle]{stealth}}}}]
\tikzstyle reverse directed=[postaction={decorate,decoration={markings,mark=at position .65 with {\arrowreversed[arrowstyle]{stealth};}}}]
\newtheorem{theorem}{Theorem} 
\newtheorem{remark}{Remark}
\newtheorem{lemma}{Lemma}
\newtheorem{assumption}{Assumption}

\newcommand{\Rmnum}[1]{\expandafter\@slowromancap\romannumeral #1@}

\newcommand{\R}{\mathbb{R}}
\newcommand{\C}{\mathbb{C}}

\newcommand{\bA}{\boldsymbol{A}} 
\newcommand{\bB}{\boldsymbol{B}} 
\newcommand{\bH}{\boldsymbol{H}}
\newcommand{\bV}{\boldsymbol{V}}  
    
\newcommand{\bE}{\boldsymbol{E}}  
\newcommand{\bn}{\boldsymbol{n}}
\newcommand{\bt}{\boldsymbol{t}}  
  
\newcommand{\cT}{\mathcal{T}}

\newcommand{\cE}{\mathcal{E}}  
\usepackage{url}
\usepackage{xcolor}
\usepackage{geometry}
\begin{document}

\title{An efficient iterative method for dynamical Ginzburg-Landau equations}

\author{Qingguo Hong}
\address{Department of Mathematics, Pennsylvania State University, University Park, PA, 16802, USA. huq11@psu.edu} 
\author{Limin Ma}
\address{Department of Applied Mathematics, The Hong Kong Polytechnic University, Hung Hom, Hong Kong. maliminpku@gmail.com}

\author{Jinchao Xu }
\address{Department of Mathematics, Pennsylvania State University, University Park, PA, 16802, USA. xu@math.psu.edu} 

\thanks{The authors gratefully acknowledge the support by the Computational Materials Sciences Program funded by the U.S. Department of Energy, Office of Science, Basic Energy Sciences, under Award No. DE-SC0020145.}

\maketitle

\begin{abstract}
In this paper, we propose a new finite element approach to simulate the time dependent  Ginzburg-Landau equations under the temporal gauge, and design an efficient preconditioner for the Newton iteration of the resulting discrete system. The new approach solves the magnetic  potential in $H({\rm curl})$ space by the lowest order of the  second kind ${\rm N\acute{e}d\acute{e}lec}$ element. 
This approach offers a simple way to deal with the boundary condition, and leads to a stable and reliable performance when dealing the superconductor with reentrant corners. 
The comparison in numerical simulations verifies the efficiency of the proposed preconditioner, which can significantly speed up the simulation in large scale computations. 


 \vskip 15pt
\end{abstract}

\section{Introduction}
\label{intro}  

The Ginzburg-Landu theory of superconductivity \cite{ginzburg1950theory} describes the transient behavior and vortex motions of superconductors in an external magnetic field. The time dependent Ginzburg-Landau (TDGL) equations  are widely used in the simulations, where the nondimensionalization form is
\begin{equation}\label{model0}
\left\{
\begin{aligned} 
\left(\partial_{t}+i\kappa  \phi\right) \psi
+ \left(\frac{i}{\kappa} \nabla + \boldsymbol{A}\right)^{2} \psi
+ (|\psi|^{2} -1) \psi &= 0& \quad \mbox{in }\ \Omega\times (0,T],\\
\sigma\left(\nabla \phi + \partial_{t} \bA\right)
+ \nabla \times(\nabla \times \boldsymbol{A})
+ Re\left[  \psi^*({i\over \kappa}\nabla + \bA)\psi  \right]
&=\nabla\times \bH& \quad \mbox{in }\ \Omega\times (0,T],
\end{aligned}
\right.
\end{equation}
with the boundary and initial conditions
\begin{equation}\label{model0bc}
\left\{
\begin{aligned} 
(\nabla \times \boldsymbol{A})\times \boldsymbol{n}&=\boldsymbol{H}\times \boldsymbol{n},
&\quad (\frac{i}{\kappa} \nabla + \boldsymbol{A}) \psi\cdot \boldsymbol{n} &=0&\text { on } \partial \Omega,
\\
\psi(x,0) &=\psi_0(x),&\quad \bA(x,0)&=\bA_0(x)&\mbox{ on } \Omega.
\end{aligned}
\right.
\end{equation}
Here $\Omega$ is a bounded domain in $\mathbb{R}^d$ $(d=2, 3)$, the order parameter $\psi$ is a complex scalar function which describes the macroscopic state of the superconductor, $\phi$ is a real scalar-valued electric potential, $\bA$ is a real vector-valued magnetic potential and the real vector-valued function $\bH$ is the external magnetic field. Variables of physical interest in this model are the superconducting density $|\psi|^2$, the magnetic induction field $\bB=\nabla\times \bA$ and the electric field $\bE=\partial_t \bA+\nabla\phi$. The total current $\boldsymbol{J}=\nabla\times \bB$, and the supercurrent
\begin{equation}\label{jsdef}
\boldsymbol{j}_s=\frac{1}{2 i \kappa}(\psi^* \nabla \psi-\psi \nabla \psi^*)-|\psi|^{2} \boldsymbol{A} = -Re\left[  \psi^*({i\over \kappa}\nabla + \bA)\psi  \right].
\end{equation}
In the nondimensionalization form \eqref{model0}, the magnitude of the order parameter $|\psi|$ is between $0$ and $1$, where $|\psi|=0$ corresponds to the normal state, $|\psi|=1$ corresponds to the superconducting state, and $0<|\psi|<1$ corresponds to some intermediate state.

The solution of the nondimensionalization model \eqref{model0} is not unique. Given any solution $(\psi, \bA, \phi)$, a gauge transformation
\begin{equation}
G_\chi (\psi, \bA, \phi)=(\psi e^{i\kappa \chi}, \bA + \nabla \chi, \phi - \partial_t \chi)
\end{equation}
gives a class of equivalent solutions, in the sense that the physical variables are invariant under gauge transformation, say superconducting density $|\psi|$, magnetic induction $\bB$ and electric field~$\bE$. 
Mathematically speaking, the solutions of \eqref{model0} under different gauges are theoretically equivalent. But numerical schemes under different gauges are computationally different. The dependence of the system on the electric potential is eliminated via a gauge transformation.  There are several widely used gauges, including the Lorentz gauge $\phi=-\nabla\cdot A$ and the temporal gauge $\phi=0$ which is considered in this paper. The equations for $\psi$ and $\bA$ are uniformly parabolic under the Lorentz gauge, some analysis was presented in \cite{chen1993non,li2017global} requiring some strong regularity of the solution and the smoothness of the domain. Many numerical methods were produced and studied in literature, see \cite{chen1997mixed,gao2014optimal,du2005numerical} and the reference therein. Some mixed element methods were proposed for the Lorentz gauge to get rid of the spurious vortex pattern by conventional methods, see \cite{gao2015efficient,gao2016new,gao2017efficient,li2017mathematical,li2020hodge,chen1997mixed}.

The TDGL equations under the temporal gauge gain more interest in the physical and engineering community \cite{alstrom2011magnetic,gropp1996numerical,mu1998alternating,richardson2004numerical,vodolazov2003vortex,winiecki2002fast,gao2019stabilized}. The nondimensionalization system under the temporal gauge solves 
\begin{equation}\label{model0temporal}
\left\{
\begin{aligned} 
\partial_{t}\psi
+ \left(\frac{i}{\kappa} \nabla + \boldsymbol{A}\right)^{2} \psi
+ (|\psi|^{2} -1) \psi &= 0& \quad \mbox{in }\ \Omega\times (0,T],
\\
\sigma \partial_{t} \bA
+ \nabla \times(\nabla \times \boldsymbol{A})
+ Re\left[  \psi^*({i\over \kappa}\nabla + \bA)\psi  \right]
&=\nabla\times \bH& \quad \mbox{in }\ \Omega\times (0,T],
\end{aligned}
\right.
\end{equation}
with the boundary and initial conditions \eqref{model0bc}. 
The system under temporal gauge looks simpler than that under Lorentz gauge, but the equation involving the magnetic potential $\boldsymbol{A}$ is no longer coercive in $H^1(\Omega)$, which in turn leads to some difficulties in designing numerically convergent schemes for the TDGL equations.  The regularity of the solutions  of the Ginzburg-Landau equations under temporal gauge was analyzed in \cite{du1992analysis,yang1989existence} on smooth domain. 
Some finite element schemes and mixed element schemes of this problem in $H^1(\Omega)$ with an additional boundary condition $\bA\cdot \bn|_{\partial \Omega}=0$ were proposed and analyzed in \cite{gao2016new,chen2001adaptive,mu1997linearized,yang2014convergence,mu1998alternating,du1994finite,wu2018analysis} and the references therein.
In a domain with reentrant corners, well-posedness of the TDGL equations and convergence of the numerical solutions are still open. 

In this paper, we propose a new nonlinear approach to solve the TDGL equations in $H({\rm curl},\Omega, \R^d)\times H^1(\Omega,\C)$ and also an efficient preconditioner for the Newton iteration solving the nonlinear system. The conventional finite element scheme with discrete approximation $\bA_h\in H^1(\Omega, \R^d)$ may lead to unstable or spurious numerical phenomenon when the regularity of solution is low, and the construction of the discrete space is not easy to implement due to the additional boundary condition. The proposed approach is more stable in this case as showed in numerical tests, and the boundary condition will not be an issue. 
The proposed scheme is a nonlinear system, which couples  two variables. The nonlinearity offers the advantage to analyze the energy decaying property of the numerical solution. The Newton method is applied to solve the nonlinear system and a preconditioner is proposed for the linearized system, where the efficiency of this preconditioner is verified by numerical tests. This efficient preconditioner plays an  important role in speeding up the simulation and makes the computational cost of this nonlinear system comparable to that of a linear system.

%

The remaining paper is organized as follows. Later in this section, some notations are introduced. Section \ref{sec:TDGL} proposes a new approach to solve the TDGL equations under the temporal gauge. Section~\ref{sec:Newton} proposes an efficient preconditioner for the Newton iteration of the nonlinear discrete system. Section \ref{sec:numerical} presents an artificial problem with exact solution to test the accuracy of the numerical scheme and some numerical examples of vortex simulations on different domains.

\section{A new approach for time dependent Ginzburg-Landau equation}\label{sec:TDGL}


Given a spatial finite element mesh $\cT_h$, let $P_r(K, \R^d)$ be the space of all polynomials of degree not greater than $r$ on any element $K$ of $\cT_h$.
Define the discrete space of the lowest order of the second kind ${\rm N\acute{e}d\acute{e}lec}$ element by 
\begin{equation}\label{spaceQ}
Q_h=\{\tilde \bA_h\in H({\rm curl}, \Omega, \R^d): \tilde \bA_h|_K \in P_1(K,\R^d), \ \int_e \tilde \bA_h\cdot \bt_e\,ds \ \mbox{ is continuous on any }\ e\in \cE_h\},
\end{equation}
and  the discrete space of the conforming linear element by $\bV_h$ 
\begin{equation}\label{spaceV}
\bV_h=\{\tilde \psi_h\in H^1(\Omega, \C): \tilde \psi_h|_K \in P_1(K,\C), \  \tilde \psi_h \ \mbox{ is continuous on any }\ e\in \cE_h\},
\end{equation}
where $\bt_e$ is the tangential direction of the edge $e$, 
$
H^1(\Omega,\mathbb{C})=\{u+iv: u, v\in H^1(\Omega,\mathbb{R})\}
$
is the Sobolev spaces for complex-valued functions, and 
$
H({\rm curl},\Omega, \mathbb{R}^d)=\{\tilde \bA: \tilde \bA\in L^2(\Omega,\R^d), \nabla\times \tilde \bA\in L^2(\Omega,\R^d)\}.
$
Denote the inner product in $L^2(G,\mathbb{C}^d)$ by 
$
(u, v)=\int_G u\cdot v^*\,dx, 
$
where $v^*$ is the conjugate of the complex function $v\in L^2(G,\mathbb{C}^d)$.


Consider the temporal gauged TDGL equations~\eqref{model0temporal} with boundary condition~\eqref{model0bc}  
\begin{equation}
\left\{
\begin{aligned} 
\partial_{t}\psi
+ \left(\frac{i}{\kappa} \nabla + \boldsymbol{A}\right)^{2} \psi
+ (|\psi|^{2} -1) \psi &= 0& \quad \mbox{in }\ \Omega\times (0,T],
\\
\sigma \partial_{t} \bA
+ \nabla \times(\nabla \times \boldsymbol{A})
+ Re\left[  \psi^*({i\over \kappa}\nabla + \bA)\psi  \right]
&=\nabla\times \bH& \quad \mbox{in }\ \Omega\times (0,T],
\\
(\nabla \times \boldsymbol{A})\times \boldsymbol{n}&=\boldsymbol{H}\times \boldsymbol{n} &\text { on } \partial \Omega\times (0,T],
\\
(\frac{i}{\kappa} \nabla + \boldsymbol{A}) \psi\cdot \boldsymbol{n} &=0&\text { on } \partial \Omega\times (0,T].
\end{aligned}
\right.
\end{equation}
Multiply the two equations in \eqref{model0temporal} by any $\tilde \bA\in H({\rm curl}, \Omega, \R^d)$ and $\tilde \psi\in H^1(\Omega, \C)$, respectively, and take the integration on the domain. With the Neumann boundary condition  \eqref{model0bc}, the weak formulation of  \eqref{model0temporal} solves $(\bA, \psi)\in H({\rm curl}, \Omega, \R^d)\times H^1(\Omega, \C)$ such that
\begin{equation}\label{weakGL}
\begin{aligned} 
(\partial_t\psi, \tilde \psi) + ({i\over \kappa}\nabla \psi + \boldsymbol{A}\psi, {i\over \kappa} \nabla \tilde \psi + \boldsymbol{A}\tilde\psi) 
+ ( (|\psi|^2-1)\psi,\tilde \psi)&=0,
\\
(\sigma\partial_t\bA, \tilde \bA) + (\nabla\times \boldsymbol{A}, \nabla \times \tilde{\boldsymbol{A}}) + (|\psi|^2\boldsymbol{A}, \tilde{\boldsymbol{A}}) 
+ ({i\over 2\kappa}(\psi^*\nabla \psi - \psi\nabla\psi^*), \tilde{\boldsymbol{A}})&=(\bH,  \nabla\times\tilde{\boldsymbol{A}}).
\end{aligned}
\end{equation} 
This implies that if the solution  $(\bA, \psi)$ of  the TDGL equations \eqref{model0temporal} belongs to the space $H({\rm curl}, \Omega, \R^d)\times H^1(\Omega, \C)$, the solution satisfies \eqref{weakGL} for any $(\tilde \bA, \tilde \psi)\in H({\rm curl}, \Omega, \R^d)\times H^1(\Omega, \C)$. 

The semi-discrete scheme for the TDGL equations \eqref{model0temporal} seeks
$(\bA_h,\psi_h)\in Q_h\times \bV_h$ such that 
\begin{equation} \label{discreteGL}
\left\{
\begin{aligned} 
(\partial_t\bA_h, \tilde \bA_h) + F_A(\bA_h, \psi_h; \tilde \bA_h, \tilde \psi_h) &= 0,
\\
(\sigma\partial_t\psi_h, \tilde \psi_h) + F_\psi(\bA_h, \psi_h;\tilde \bA_h, \tilde \psi_h)&=0,
\end{aligned}
\right.
\quad \forall (\tilde \bA_h,\tilde \psi_h)\in Q_h\times \bV_h,
\end{equation} 
where $Q_h\subset H({\rm curl}, \Omega, \R^d)$ is the real-valued space defined in \eqref{spaceQ}, $\bV_h\subset H^1(\Omega, \C)$ is the complex-valued space defined in \eqref{spaceV}, and
\begin{equation}\label{Fdef}
\begin{aligned}
F_A(\bA_h, \psi_h; \tilde \bA_h, \tilde \psi_h) &= (\nabla\times \boldsymbol{A}_h, \nabla \times \tilde{\boldsymbol{A}}_h) 
+ (|\psi_h|^2\boldsymbol{A}_h, \tilde{\boldsymbol{A}}_h) 
+ ({i\over 2\kappa}(\psi_h^*\nabla \psi_h - \psi_h\nabla\psi_h^*), \tilde{\boldsymbol{A}}_h)
- (\bH,  \nabla\times\tilde{\boldsymbol{A}}_h),
\\
F_\psi(\bA_h, \psi_h; \tilde \bA_h, \tilde \psi_h) &=  ({i\over \kappa}\nabla \psi_h + \bA_h\psi_h, {i\over \kappa} \nabla \tilde \psi_h + \bA_h\tilde\psi_h) 
+ ( (|\psi_h|^2-1)\psi_h,\tilde \psi_h).
\end{aligned}
\end{equation}

Note that the conventional finite element solves the magnetic potential $\bA_h$ of  \eqref{weakGL} in  a smaller space
$$
Q_h^o=\{\tilde \bA_h\in H^1(\Omega, \R^d): \tilde \bA_h|_K\in P_1(K,\R^d),\ \tilde \bA_h\cdot \bn =0\},
$$
which is a subspace of $H^1(\Omega,\mathbb{R}^d)$. The additional constraint $\tilde \bA_h\cdot \bn =0$ adds the difficulty in the construction of the discrete space $Q_h^o$. Many mixed formulations and some methods based on Hodge decomposition, which introduce some extra variables, are proposed in literature to avoid the difficulty.  Compared to the conventional finite element method, the semi-discrete scheme \eqref{discreteGL} seeks $\bA_h$ in a $H({\rm curl})$-conforming finite element space and the resulting formulation is easy to implement with no difficulty in the construction of the discrete space. On the other hand, the semi-discrete scheme \eqref{discreteGL} requires weaker regularity on the solution than the conventional method which seeks the solution in a smaller space 
$$
H^1_{\bn}(\Omega,\mathbb{R})=\{\bA\in H^1(\Omega,\mathbb{R}): \bA\cdot \bn|_{\partial \Omega}=0\}.
$$ 
If the solution of the TDGL equations \eqref{model0temporal} is smooth enough and in $H^1_{\bn}(\Omega, \R^d)$, then it also satisfies the weak formulation \eqref{weakGL}.
However, the superconductor can be not convex and the solution  not smooth enough. For this case, the solution of the scheme \eqref{discreteGL} can behave better than an approximate solution in a discrete subspace of $H^1_{\bn}(\Omega, \R^d)$.

%

Let 
$
0<t_0<t_1<\cdots<t_N=T
$
be a uniform partition of the time interval with step size $\triangle t=\frac{T}{N}$. By applying the backward Euler method for time discretization, we propose a new approach for the TDGL equations \eqref{model0temporal}, which seeks $(\bA^{n+1}_h,\psi^{n+1}_h)\in Q_h\times \bV_h$ such that 
\begin{equation} \label{Eulerh2}
\left\{
\begin{aligned} 
({\psi^{n+1}_h-\psi^n_h\over \triangle t}, \tilde \psi_h) 
+ F_\psi(\bA^{n+1}_h,\psi^{n+1}_h;\tilde \bA_h, \tilde \psi_h)&= 0,
\\
(\sigma { \bA^{n+1}_h-\bA^n_h\over \triangle t},{\tilde \bA}_h) 
+ F_A(\bA^{n+1}_h,\psi^{n+1}_h; \tilde \bA_h, \tilde \psi_h)&=0,
\end{aligned}
\right.
\quad \forall (\tilde \bA_h,\tilde \psi_h)\in Q_h\times \bV_h,
\end{equation} 
and $(\bA^{0}_h,\psi^{0}_h)$ are the projections of $\bA_0$ and $\psi_0$ into $Q_h$ and $\bV_h$, respectively, namely
\begin{equation}\label{initialpro}
\begin{aligned}
(\nabla\times \bA^{0}_h,\nabla\times \tilde \bA_h) + (\bA^{0}_h, \tilde\bA_h)
&=(\nabla\times \bA^{0},\nabla\times \tilde \bA_h) + (\bA^{0}, \tilde\bA_h), 
&\quad \forall \tilde \bA_h\in Q_h,
\\
(\nabla  \psi^{0}_h,\nabla\times \tilde \psi_h) + (\psi^{0}_h, \tilde\psi_h)
&=(\nabla\times \psi^{0},\nabla\times \tilde \psi_h) + (\psi^{0}, \tilde\psi_h), 
&\quad \forall \tilde \psi_h\in \bV_h.
\end{aligned}
\end{equation}

For any $\triangle t>0$ and $n\ge 0$, define
$$
\mathcal{J}_h^{n+1}(\bB_h, \xi_h) = \mathcal{G}(\bB_h, \xi_h) + \frac{1}{\triangle t}(\|\xi_h - \psi_h^n\|_0^2 + \sigma\|\bB_h - \bA_h^n\|_0^2),\quad \forall (\bB_h, \xi_h)\in Q_h\times \bV_h,
$$
where the Gibbs free energy $\mathcal{G}(\bB, \xi)$ is defined by 
\begin{equation}\label{energydef}
\mathcal{G}(\bB, \xi)=\left\| ({i\over \kappa}\nabla  + \boldsymbol{B})\psi\right\|^2_0 + \frac12 \||\xi|^2-1\|^2_0 + \|\nabla\times\bB - \bH\|_0^2.
\end{equation} 
The semi-discrete scheme \eqref{weakGL} is a gradient flow of this free energy functional $\mathcal{G}(\bA, \psi)$, that is 
\begin{equation} 
\begin{aligned} 
(\partial_t\psi, \tilde \psi) = -( \mathcal{G}_{\psi}(\bA, \psi),\tilde \psi)&,\quad
(\sigma \partial_t\bA,\tilde \bA)  = -( \mathcal{G}_{\bA}(\bA, \psi),\tilde \bA),
\end{aligned}
\end{equation} 
where $\mathcal{G}_{\psi}$ and $\mathcal{G}_{\bA}$ are the Frech$\rm \acute{e}$t derivatives of $\mathcal{G}(\bA, \psi)$.

Note that the functional $\mathcal{J}_h^{n+1}(\cdot,\cdot)$ is nonnegative and has at least one minimum. For any $n\ge 0$, let $(\bar\bA^{n+1}_h, \bar\psi^{n+1}_h)\in Q_h\times\bV_h$ be a minimum of $\mathcal{J}_h^{n+1}(\cdot,\cdot)$ satisfying
$$
\mathcal{J}_h^{n+1}(\bar\bA^{n+1}_h, \bar\psi^{n+1}_h)=\min_{(\bB_h,\xi_h)\in Q_h\times\bV_h } \mathcal{J}_h^{n+1}(\bB_h,\xi_h).
$$
Note that $(\bar\bA^{n+1}_h, \bar\psi^{n+1}_h)$ is a solution of \eqref{Eulerh2} and
\begin{equation}\label{modifyenergydecay}
\mathcal{J}_h^{n+1}(\bar\bA_h^{n+1}, \bar\psi_h^{n+1}) \le \mathcal{J}_h^{n+1}(\bA_h^n, \psi_h^n) = \mathcal{G}(\bA_h^n, \psi_h^n).
\end{equation} 
This indicates that there exists at least one solution of the nonlinear scheme \eqref{Eulerh2}.
A similar  analysis  to the one in \cite{du1994finite} shows that the functional $\mathcal{J}_h^{n+1}(\cdot,\cdot)$ is convex in the set 
$$
\mathcal{M}=\left\{(\bA_h,\psi)\in Q_h\times \bV_h: \|\psi\|_{0,4}\le C,\ \|\bA_h\|_0\le C,\ \|({i\over \kappa}\nabla  + \boldsymbol{A}_h)\psi_h\|_0\le C\right\}
$$
with any constant $C>0$ if $\Delta t$ and $\Delta t h^{-d/2}$ are sufficiently small.  This implies that the solution of the nonlinear system \eqref{Eulerh2} is unique if $\Delta t h^{-d/2}$ are sufficiently small.

The inequality \eqref{modifyenergydecay} directly leads to the following energy decay property of a solution of the the nonlinear scheme \eqref{Eulerh2}.
\begin{theorem}\label{th:decay}
For any $\triangle t>0$ and $n\ge 0$, there exists a solution $(\bA^{n+1}_h, \psi^{n+1}_h)$ of problem \eqref{Eulerh2} such that
\[
\mathcal{G}(\bA^{n+1}_h, \psi^{n+1}_h) + \frac{1}{\triangle t}\sum_{i=0}^{n}\left(\|\psi_h^{i+1} - \psi_h^i\|_0^2 + \|\bA_h^{i+1} - \bA_h^i\|_0^2\right)\le \mathcal{G}(\bA^{0}_h, \psi^{0}_h).
\]
Especially, if $\Delta t h^{-d/2}$ are sufficiently small, the solution of the nonlinear problem \eqref{Eulerh2} is unique and admits the energy decay property above .
\end{theorem}

If the time step $\Delta t h^{-d/2}$ is not small enough, 
 the nonlinear system \eqref{Eulerh2} may have multiple solutions. 
The following theorem proves that the discrete energy is bounded for any solution of \eqref{Eulerh2} under some suitable condition.

According to \cite{chen1993non},  $|\psi|\le 1$ holds for the solution $\psi$ of problem \eqref{model0temporal} at any time $t> 0$ if the initial condition $|\psi_0|\le 1$. Although it is not easy to prove that the numerical scheme preserves this property, numerical tests in Section~\ref{sec:numerical} show that the solution of the proposed scheme \eqref{Eulerh2} satisfies $|\psi_h^n|\le 1$ numerically. Thus we can reasonably assume that the discrete order parameter $\psi_h$ is bounded as shown below.

\begin{assumption}\label{ass:bound}
There exists a positive constant $C$ such that  for any $1\le i\le N$,
\begin{equation}
\|\psi_h^i\|_\infty\le C,
\end{equation}
where  $(\bA^{i}_h, \psi^{i}_h)$ be the solution of \eqref{Eulerh2}  with $\triangle t=\frac1N$.
\end{assumption}

\begin{theorem}\label{energy:estimate}
Let $(\bA^{n+1}_h, \psi^{n+1}_h)$ be the solution of \eqref{Eulerh2}  with $\triangle t=\frac1N$ with Assumption \ref{ass:bound} holds. If $\triangle t<\min\{\frac12, {\sigma \over 4\max_{1\le i\le n} \|\psi_h^i\|_\infty^2}\}$ holds,
there exists a positive constant $\tilde C$, which is independent of $N$, such that  
\[
\mathcal{G}(\bA^{n+1}_h, \psi^{n+1}_h) + \frac{1}{\triangle t}\sum_{i=0}^{n}\left(\|\psi_h^{i+1} - \psi_h^i\|_0^2 + \sigma\|\bA_h^{i+1} - \bA_h^i\|_0^2\right)\le \tilde C\mathcal{G}(\bA^{0}_h, \psi^{0}_h),\qquad \forall n\le N.
\]
\end{theorem}
\begin{proof}  
For any $a$, $b\in \mathbb{C}$, it holds that
\begin{equation}\label{basicineq}
|a|^2 - |b|^2=-|a-b|^2 + 2{\rm Re}(a,a-b).
\end{equation}
It follows that
\begin{equation}\label{abterms2}
\begin{aligned}
\left\| ({i\over \kappa}\nabla  + \boldsymbol{A}^{i+1}_h)\psi^{i+1}_h\right\|^2_0
-\left\| ({i\over \kappa}\nabla  + \boldsymbol{A}^{i+1}_h)\psi^{i}_h\right\|^2_0
=&2{\rm Re}\left(({i\over \kappa}\nabla  + \boldsymbol{A}^{i+1}_h)\psi^{i+1}_h,({i\over \kappa}\nabla  + \boldsymbol{A}^{i+1}_h)(\psi^{i+1}_h - \psi^{i}_h)\right)
\\
&-\left\| ({i\over \kappa}\nabla  + \boldsymbol{A}^{i+1}_h)(\psi^{i+1}_h - \psi^{i}_h)\right\|^2_0,
\end{aligned}
\end{equation}
\begin{equation}\label{abterms21} 
\left\| ({i\over \kappa}\nabla  + \boldsymbol{A}^{i+1}_h)\psi^{i}_h\right\|^2_0
- \left\| ({i\over \kappa}\nabla  + \boldsymbol{A}^{i}_h)\psi^{i}_h\right\|^2_0
=2{\rm Re}\left(({i\over \kappa}\nabla  + \boldsymbol{A}^{i+1}_h)\psi^{i}_h, (\boldsymbol{A}^{i+1}_h - \boldsymbol{A}^{i}_h) \psi^{i}_h\right)
-\left\| (\boldsymbol{A}^{i+1}_h - \boldsymbol{A}^{i}_h) \psi^{i}_h\right\|^2_0 .
\end{equation}
Note that the first term on the right-hand side of the above equation can be decomposed as 
\begin{equation}\label{abterms22}
\begin{aligned}
2Re(({i\over \kappa}\nabla  + \boldsymbol{A}^{i+1}_h)\psi^{i}_h, (\boldsymbol{A}^{i+1}_h - \boldsymbol{A}^{i}_h)\psi^{i}_h)
=&2Re\left(({i\over \kappa}\nabla  + \boldsymbol{A}^{i+1}_h)\psi^{i+1}_h, (\boldsymbol{A}^{i+1}_h - \boldsymbol{A}^{i}_h)\psi^{i+1}_h\right)
\\
&- 2Re\left(({i\over \kappa}\nabla  + \boldsymbol{A}^{i+1}_h)(\psi^{i+1}_h - \psi^{i}_h), (\boldsymbol{A}^{i+1}_h - \boldsymbol{A}^{i}_h)\psi^{i+1}_h\right) 
\\
&- 2Re\left(({i\over \kappa}\nabla  + \boldsymbol{A}^{i+1}_h)\psi^{i}_h, (\boldsymbol{A}^{i+1}_h - \boldsymbol{A}^{i}_h)(\psi^{i+1}_h - \psi^{i}_h)\right).
\end{aligned}
\end{equation}
An addition of \eqref{abterms2}, \eqref{abterms21} and \eqref{abterms22} yields
\begin{equation}\label{adelta20}
\begin{aligned}
\left\| ({i\over \kappa}\nabla  + \boldsymbol{A}^{i+1}_h)\psi^{i+1}_h\right\|^2_0
=&\left\| ({i\over \kappa}\nabla  + \boldsymbol{A}^{i}_h)\psi^{i}_h\right\|^2_0
-\|(\boldsymbol{A}^{i+1}_h - \boldsymbol{A}^{i}_h)\psi_h^i\|_0^2
-\left\| ({i\over \kappa}\nabla  + \boldsymbol{A}^{i+1}_h)(\psi^{i+1}_h - \psi^{i}_h)\right\|^2_0
\\
&+2{\rm Re}(({i\over \kappa}\nabla  + \boldsymbol{A}^{i+1}_h)\psi^{i+1}_h,({i\over \kappa}\nabla  + \boldsymbol{A}^{i+1}_h)(\psi^{i+1}_h - \psi^{i}_h))
\\ 
&+ 2{\rm Re}(({i\over \kappa}\nabla  + \boldsymbol{A}^{i+1}_h)\psi^{i+1}_h, (\boldsymbol{A}^{i+1}_h - \boldsymbol{A}^{i}_h)\psi^{i+1}_h)
\\
&- 2Re\left(({i\over \kappa}\nabla  + \boldsymbol{A}^{i+1}_h)(\psi^{i+1}_h - \psi^{i}_h), (\boldsymbol{A}^{i+1}_h - \boldsymbol{A}^{i}_h)\psi^{i+1}_h\right) 
\\
&- 2Re\left(({i\over \kappa}\nabla  + \boldsymbol{A}^{i+1}_h)\psi^{i}_h, (\boldsymbol{A}^{i+1}_h - \boldsymbol{A}^{i}_h)(\psi^{i+1}_h - \psi^{i}_h)\right).
\end{aligned}
\end{equation} 
By Young's inequality, we have
\begin{equation}\label{adelta201}
\begin{aligned}
\left|2Re\left(({i\over \kappa}\nabla  + \boldsymbol{A}^{i+1}_h)(\psi^{i+1}_h - \psi^{i}_h), (\boldsymbol{A}^{i+1}_h - \boldsymbol{A}^{i}_h)\psi^{i+1}_h\right)\right|
\le & \left\|({i\over \kappa}\nabla  + \boldsymbol{A}^{i+1}_h)(\psi^{i+1}_h - \psi^{i}_h)\right\|_0^2 + \|(\boldsymbol{A}^{i+1}_h - \boldsymbol{A}^{i}_h)\psi_h^{i+1}\|_0^2
\\
\left |2Re\left(({i\over \kappa}\nabla  + \boldsymbol{A}^{i+1}_h)\psi^{i}_h, (\boldsymbol{A}^{i+1}_h - \boldsymbol{A}^{i}_h)(\psi^{i+1}_h - \psi^{i}_h)\right)\right|
\le& \epsilon \left\| ({i\over \kappa}\nabla  + \boldsymbol{A}^{i}_h)\psi^{i}_h\right\|^2_0
+ {2\over \epsilon}\|(\boldsymbol{A}^{i+1}_h - \boldsymbol{A}^{i}_h)  \psi^{i+1}_h \|_0^2+{2\over \epsilon}\|(\boldsymbol{A}^{i+1}_h - \boldsymbol{A}^{i}_h)  \psi^{i}_h \|_0^2
\\
&+ \left|2Re\left((\bA_h^{i+1}-\bA_h^i)\psi_h^i, (\bA_h^{i+1}-\bA_h^i)(\psi_h^{i+1}-\psi_h^i)\right)\right|.
\end{aligned}
\end{equation} 
Let $C=\max_{1\le i\le n} \|\psi_h^i\|_\infty$. A substitution of \eqref{adelta201} into \eqref{adelta20}  yields
\begin{equation}\label{adelta2}
\begin{aligned}
\left\| ({i\over \kappa}\nabla  + \boldsymbol{A}^{i+1}_h)\psi^{i+1}_h\right\|^2_0
\le&(1+\epsilon) \left\| ({i\over \kappa}\nabla  + \boldsymbol{A}^{i}_h)\psi^{i}_h\right\|^2_0
+2{\rm Re}\left(({i\over \kappa}\nabla  + \boldsymbol{A}^{i+1}_h)\psi^{i+1}_h,({i\over \kappa}\nabla  + \boldsymbol{A}^{i+1}_h)(\psi^{i+1}_h - \psi^{i}_h)\right)
\\ 
&+ 2{\rm Re}\left(({i\over \kappa}\nabla  + \boldsymbol{A}^{i+1}_h)\psi^{i+1}_h, (\boldsymbol{A}^{i+1}_h - \boldsymbol{A}^{i}_h)\psi^{i+1}_h\right)
+ 4C^2(1+\frac{1}{\epsilon})\|(\boldsymbol{A}^{i+1}_h - \boldsymbol{A}^{i}_h) \|_0^2.
\end{aligned}
\end{equation} 
It follows \eqref{basicineq} that 
\begin{equation}\label{psidelta2}
\begin{aligned}
\left \||\psi^{i+1}_h|^2-1\right \|^2_0 - \left \||\psi^{i}_h|^2-1\right \|^2_0
= &4{\rm Re}((|\psi_h^{i+1}|^2-1)\psi_h^{i+1}, \psi_h^{i+1} - \psi_h^{i})
-\left\||\psi_h^{i+1}|^2 - |\psi_h^{i}|^2\right\|_0^2
\\
&-2\int_{\Omega} (|\psi_h^{i+1}|^2-1)|\psi_h^{i+1} - \psi_h^{i}|^2\,dx,
\end{aligned}
\end{equation}
\begin{equation}\label{curladelta2}
\begin{aligned}
\|\nabla\times \bA^{i+1}_h - \bH \|^2_0 - \|\nabla\times \bA^{i}_h - \bH \|^2_0
= &
-\|\nabla\times (\bA^{i+1}_h - \bA^{i}_h)\|^2_0
+ 2\left(\nabla\times \bA^{i+1}_h-\bH, \nabla\times (\bA^{i+1}_h - \bA^{i}_h)\right).
\end{aligned}
\end{equation}
A combination of \eqref{adelta2}, \eqref{psidelta2} and \eqref{curladelta2} gives
\begin{equation}\label{energyfinal1}
\begin{aligned}
\mathcal{G}(\bA^{i+1}_h, \psi^{i+1}_h)  - (1+\epsilon)\mathcal{G}(\bA^{i}_h, \psi^{i}_h)  
\le & 2{\rm Re}\left(({i\over \kappa}\nabla  + \boldsymbol{A}^{i+1}_h)\psi^{i+1}_h,({i\over \kappa}\nabla  + \boldsymbol{A}^{i+1}_h)(\psi^{i+1}_h - \psi^{i}_h)\right)
\\ 
&+ 2{\rm Re}\left(({i\over \kappa}\nabla  + \boldsymbol{A}^{i+1}_h)\psi^{i+1}_h, (\boldsymbol{A}^{i+1}_h - \boldsymbol{A}^{i}_h)\psi^{i+1}_h\right)
\\
&+2{\rm Re}\left((|\psi_h^{i+1}|^2-1)\psi_h^{i+1}, \psi_h^{i+1} - \psi_h^{i}\right)
+ 2\|\psi_h^{i+1} - \psi_h^{i} \|_0^2
\\
&+ 2(\nabla\times \bA^{i+1}_h-\bH, \nabla\times (\bA^{i+1}_h - \bA^{i}_h))
+ ({4C^2\over \epsilon} + 4C^2)\|\boldsymbol{A}^{i+1}_h - \boldsymbol{A}^{i}_h\|_0^2.
\end{aligned}
\end{equation}
By \eqref{jsdef} and \eqref{Fdef}, the right-hand side of the above inequality can be rewritten as 
\begin{equation}\label{energyfinal2}
\begin{aligned}
&2{\rm Re}F_\psi(\bA_h^{i+1}, \psi_h^{i+1};  \bA_h^{i+1}-\bA_h^{i}, \psi_h^{i+1}-\psi_h^{i})
+ 2F_A(\bA_h^{i+1}, \psi_h^{i+1}; \bA_h^{i+1}-\bA_h^{i}, \psi_h^{i+1}-\psi_h^{i})
\\
&+ 2\|\psi_h^{i+1} - \psi_h^{i} \|_0^2
+ ({4C^2\over \epsilon} + 4C^2)\|\boldsymbol{A}^{i+1}_h - \boldsymbol{A}^{i}_h\|_0^2.
\end{aligned}
\end{equation}
It follows \eqref{Eulerh2} that
\begin{equation}\label{energyfinal3}
\begin{aligned}
F_\psi(\bA_h^{i+1}, \psi_h^{i+1};  \bA_h^{i+1}-\bA_h^{i}, \psi_h^{i+1}-\psi_h^{i})
&=- { 1\over \triangle t}\|\psi^{i+1}_h-\psi^i_h\|_0^2,
\\
F_{A}(\bA_h^{i+1}, \psi_h^{i+1}; \bA_h^{i+1}-\bA_h^{i}, \psi_h^{i+1}-\psi_h^{i})
&=- { \sigma \over \triangle t}\|\bA^{i+1}_h-\bA^i_h\|_0^2.
\end{aligned}
\end{equation}
A combination of \eqref{energyfinal1}, \eqref{energyfinal2} and \eqref{energyfinal3} yields
\[
\begin{aligned}
&\mathcal{G}(\bA^{i+1}_h, \psi^{i+1}_h)  - (1+\epsilon)\mathcal{G}(\bA^{i}_h, \psi^{i}_h) + {\sigma\over \Delta t}\|\bA_h^{i+1}-\bA_h^{i}\|_0^2 + {1\over \triangle t}\|\psi^{i+1}_h-\psi^i_h\|_0^2
\\
\le & - ({ \sigma \over \triangle t} - {4C^2\over \epsilon} - 4C^2 )\|\bA^{i+1}_h-\bA^i_h\|_0^2
- ({1\over \triangle t} - 2)\|\psi^{i+1}_h-\psi^i_h\|_0^2.
\end{aligned}
\]
Choose $\epsilon= {4C^2\over \sigma N-4C^2}$. The right-hand side of the above inequality is negative. Thus, for any $0\le i\le N-1$,
\begin{equation}\label{induc:0}
\mathcal{G}(\bA^{i+1}_h, \psi^{i+1}_h) + {\sigma\over \Delta t}\|\bA_h^{i+1}-\bA_h^{i}\|_0^2 + {1\over \triangle t}\|\psi^{i+1}_h-\psi^i_h\|_0^2
\le  (1+\epsilon)\mathcal{G}(\bA^{i}_h, \psi^{i}_h).
\end{equation} 
Define 
\[
T_k^{n+1}=\mathcal{G}(\bA^{n+1}_h, \psi^{n+1}_h) + {1\over \Delta t}\sum_{j=k}^n( \sigma\|\bA_h^{j+1}-\bA_h^{j}\|_0^2 + \|\psi^{j+1}_h-\psi^j_h\|_0^2).
\]
Next we prove the following result by induction
\begin{equation}\label{induc:ass}
T_k^{n+1} \le (1+\epsilon)^{n+1-k}\mathcal{G}(\bA^{k}_h, \psi^{k}_h).
\end{equation} 
By \eqref{induc:0}, the inequality  \eqref{induc:ass} with $k=n$ holds. Note that $T_{k-1}^{n+1}=T_k^{n+1} + {\sigma\over \Delta t}\|\bA_h^{k}-\bA_h^{k-1}\|_0^2 + {1\over \triangle t}\|\psi^{k}_h-\psi^{k-1}_h\|_0^2$. It follows~\eqref{induc:0} that
\[
\begin{aligned}
T_{k-1}^{n+1}
\le & (1+\epsilon)^{n+1-k}(\mathcal{G}(\bA^{k}_h, \psi^{k}_h)  + {\sigma\over \Delta t}\|\bA_h^{k}-\bA_h^{k-1}\|_0^2 + {1\over \triangle t}\|\psi^{k}_h-\psi^{k-1}_h\|_0^2)
\le (1+\epsilon)^{n+1-(k-1)}\mathcal{G}(\bA^{k-1}_h, \psi^{k-1}_h),
\end{aligned}
\]
which completes the proof for \eqref{induc:ass}.
This implies that 
\[
\mathcal{G}(\bA^{n+1}_h, \psi^{n+1}_h) + {1\over \Delta t}\sum_{j=0}^n( \sigma\|\bA_h^{j+1}-\bA_h^{j}\|_0^2 + \|\psi^{j+1}_h-\psi^j_h\|_0^2)\le (1+\epsilon)^{N}\mathcal{G}(\bA^{0}_h, \psi^{0}_h)\le \tilde C\mathcal{G}(\bA^{0}_h, \psi^{0}_h),
\]
where $\tilde C$ is a constant independent of $N$.
\end{proof}
\begin{remark}
Numerically the constant $C$ in Assumption \ref{ass:bound} is not larger than 1, which indicates that the condition on the time step $\Delta t$ in Theorem \ref{energy:estimate} can easily be satisfied.
\end{remark}
\section{Newton method and preconditioner}\label{sec:Newton}
In this section, we briefly outline the Newton method for the nonlinear system \eqref{Eulerh2} and propose an efficient preconditioner for the linearized system to speed up the Newton iteration.

\subsection{A preconditioner for the linearized system}
The Newton method is employed to solve the nonlinear system~\eqref{Eulerh2} with the initial approximation $(\bA_h^{0}, \psi_{h}^{0})$ solved by \eqref{initialpro}. Given the approximation solution $(\bA_h^n, \psi_{h}^n)$ at time step $t_n$,  denote the initial guess of the  Newton iteration at current time step by $(\bA_h^{n+1,0}, \psi_{h}^{n+1,0})=(\bA_h^n, \psi_{h}^n)$. For any given $(\bA_h^{n+1,k}, \psi_{h}^{n+1,k})$,  the residual 
$f^{n+1,k}(\tilde \bA_h, \tilde \psi_h)=\sum_{j=1}^2f^{n+1,k}_j(\tilde \bA_h, \tilde \psi_h)$ with 
 \begin{align} 
f^{n+1,k}_1(\tilde \bA_h, \tilde \psi_h)&= -({\psi^{n+1,k}_h-\psi^n_h\over \triangle t}, \tilde \psi_h) 
- F_\psi(\bA^{n+1,k}_h,\psi^{n+1.k}_h;\tilde \bA_h, \tilde \psi_h),
\\
f^{n+1,k}_2(\tilde \bA_h, \tilde \psi_h)&= -(\sigma { \bA^{n+1,k}_h-\bA^n_h\over \triangle t},{\tilde \bA}_h) 
- F_A(\bA^{n+1,k}_h,\psi^{n+1,k}_h; \tilde \bA_h, \tilde \psi_h).
\end{align}
Then Newton method generates a sequence of feasible approximate solution $(\bA_h^{n+1,k+1}, \psi_{h}^{n+1,k+1})$ such that
\begin{equation}\label{newton2vform}
\bA_h^{n+1,k+1}=\bA_h^{n+1,k} + \bB_h^{n+1,k+1},
\quad 
\psi_{h}^{n+1,k+1}= \psi_{h}^{n+1,k} + \xi_h^{n+1,k+1},
\end{equation}
where $(\bB_h^{n+1,k+1}, \xi_{h}^{n+1,k+1})$ is the solution to the linearized system
\begin{equation}\label{linear2v}
a^{n+1,k+1}(\bB_h^{n+1,k+1}, \xi_{h}^{n+1,k+1}; \tilde \bA_h, \tilde \psi_h)=f^{n+1,k}(\tilde \bA_h, \tilde \psi_h),
\end{equation}
with the bilinear form 
$$
a^{n+1,k+1}(\bB_h, \xi_h; \tilde \bA_h, \tilde \psi_h)=\sum_{j=1}^4a_j^{n+1,k+1}(\bB_h, \xi_h; \tilde \bA_h, \tilde \psi_h)
$$
where $\psi_h^{n+1,k,*}$ is the conjugate of $\psi_h^{n+1,k}$ and
\begin{equation}\label{adef}
\begin{aligned}
a_1^{n+1,k+1}(\bB_h, \xi_h; \tilde \bA, \tilde \psi_h)
=&
\frac{1}{\triangle t}  (\xi_h, \tilde \psi_h)
+ ((2|\psi_h^{n+1,k}|^2-1)\xi_h, \tilde \psi_h)
+((\psi_h^{n+1,k})^2\xi_h^*,\tilde \psi_h) 
\\
&+ (\frac{i}{\kappa}\nabla \xi_h+\bA_h^{n+1,k}\xi_h, \frac{i}{\kappa}\nabla\tilde \psi_h+\bA_h^{n+1,k}\tilde \psi_h), 
\\
a_2^{n+1,k+1}(\bB_h, \xi_h; \tilde \bA, \tilde \psi_h)
=&
(\psi_h^{n+1,k}\bB_h, \frac{i}{\kappa}\nabla\tilde \psi_h+\bA_h^{n+1,k}\tilde \psi_h)
+ \left((\frac{i}{\kappa}\nabla \psi_h^{n+1,k}+\bA_h^{n+1,k}\psi_h^{n+1,k})\cdot \bB_h, \tilde \psi_h\right),
\\
a_3^{n+1,k+1}(\bB_h, \xi_h; \tilde \bA, \tilde \psi_h)
=&  \hbox{Re}\left( \frac{i}{\kappa}\nabla\xi_h+\bA_h^{n+1,k}\xi_h, \psi_h^{n+1,k}\tilde\bA_h \right)+
\hbox{Re}\left(\xi_h, (\frac{i}{\kappa}\nabla \psi_h^{n+1,k}+\bA_h^{n+1,k}\psi_h^{n+1,k}) \cdot\tilde\bA_h\right),
\\
a_4^{n+1,k+1}(\bB_h, \xi_h; \tilde \bA, \tilde \psi_h)
=&
\frac{\sigma}{\triangle t} (\bB_h, \tilde \bA_h)+(\nabla \times \bB_h, \nabla \times \tilde\bA_h)
+ (|\psi_h^{n+1,k}|^2\bB_h, \tilde \bA_h).
\end{aligned}
\end{equation} 
The matrix form of the linear system \eqref{linear2v} can be written as 
\begin{equation}\label{amatrix}
A^{n+1,k+1}x=b,\quad \mbox{with}\quad
A^{n+1,k+1}=\begin{pmatrix}
A^{n+1,k+1}_{\xi_h,\tilde \psi_h}&A^{n+1,k+1}_{\xi_h,\tilde \bA_h}\\
A^{n+1,k+1}_{\bB_h,\tilde \psi_h}&A^{n+1,k+1}_{\bB_h,\tilde \bA_h}
\end{pmatrix},\ 
x=\begin{pmatrix}
x_{\xi_h}^{n+1,k+1}\\
x_{\bB_h}^{n+1,k+1}
\end{pmatrix},\
b=\begin{pmatrix}
b_{\tilde \psi_h}\\
b_{\tilde \bA_h}
\end{pmatrix},
\end{equation}
where $A^{n+1,k+1}_{\xi_h,\tilde \psi_h}$, $A^{n+1,k+1}_{\xi_h,\tilde \bA_h}$, $A^{n+1,k+1}_{\bB_h,\tilde \psi_h}$ and $A^{n+1,k+1}_{\bB_h,\tilde \bA_h}$ are  matrix forms of  $a_j^{n+1,k+1}$ in \eqref{adef} with $j=1$, $2$, $3$ and $4$, respectively, and $b$ is the vector form of $f^{n+1,k}$ on the right-hand side of \eqref{linear2v}.

The linear problem \eqref{linear2v} needs to be solved in each Newton iteration \eqref{newton2vform}, until a stopping criteria is satisfied. For large systems of linear problems, exact solver can be very inefficient, and an iterative method with efficient preconditioner can speed up the computation. To start with the design of the preconditioner for the linear system \eqref{linear2v}, we define an auxiliary bilinear form 
\begin{equation}\label{aP1}
\begin{split}
a_{\rm \tilde P}^{n+1,k+1}(\xi_h, \bB_h; \tilde \psi_h, \tilde \bA_h)
=&\frac{1}{\triangle t}(\xi_h, \tilde \psi_h)+  (|\psi_h^{n+1,k}|^2\xi_h, \tilde \psi_h) + (\frac{i}{\kappa}\nabla\xi_h+\bA_h^{n+1,k}\xi_h, \frac{i}{\kappa}\nabla\tilde \psi_h+\bA_h^{n+1,k}\tilde \psi_h)
\\
&+ \frac{\sigma}{\triangle t} (\bB_h, \tilde \bA_h)
+ (\nabla \times\bB_h, \nabla \times\tilde \bA_h)
+ (|\psi_h^{n+1,k}|^2\bB_h, \tilde \bA_h),
\end{split}
\end{equation}
which is derived from  the norms
\begin{equation}\label{norm1}
\begin{aligned}
\|\xi_h\|_h^2&=\frac{1}{\triangle t} \|\xi_h\|_0^2+ \|\frac{i}{\kappa}\nabla\xi_h+\bA_h^{n+1,k}\xi_h\|_0^2
+ (|\psi_h^{n+1,k}|^2\xi_h, \xi_h),
&\forall \xi_h\in\bV_h,
\\
\|\bB_h\|_h^2&=\frac{\sigma}{\triangle t} \|\bB_h\|_0^2
+ \|\nabla \times\bB_h\|_0^2
+ (|\psi_h^{n+1,k}|^2\bB_h, \bB_h),& \forall \bB_h\in Q_h. 
\end{aligned}
\end{equation} 
Note that the matrix form of this bilinear form is block diagonal and can be written as 
\begin{equation}\label{pdef1}
\tilde P^{n+1,k+1}=\begin{pmatrix}
A^{n+1,k+1}_{{\rm P}, \xi}&0\\
0&A^{n+1,k+1}_{{\rm P}, \bB}
\end{pmatrix}.
\end{equation}
The inverse of $\tilde P^{n+1,k+1}$ is much easier to compute than that of the matrix $A^{n+1,k+1}$.
Then, we apply the GMRES method with $\tilde P^{n+1,k+1}$ as the preconditioner to solve the linear problem \eqref{amatrix} on each Newton iteration.

\begin{algorithm}[H]
  \SetAlgoLined
  \KwData{Given the initial data $\psi_0$, $\bA_0$, boundary data $\bH$, time step $\Delta t={T\over N}$}
  \KwResult{Approximation $(\bA_h^N, \psi_h^N)$ at time $T=t_N$ }

  Initialization: $(\bA^{0}_h,\psi^{0}_h)$ are the projections of $\bA_0$ and $\psi_0$ into $Q_h$ and $\bV_h$ by solving \eqref{initialpro};
  
  \For{$n=1:N-1$}{
  $k=0$\;
  $(\bA^{n+1,0}_h,\psi^{n+1,0}_h)=(\bA^{n}_h,\psi^{n}_h)$\;
  $err$ is the residual of the nonlinear problem \eqref{Eulerh2}\;
  \While{$err>tol$}{
    solve the linear problem \eqref{amatrix} by GMRES method with preconditioner $\tilde P^{n+1,k+1}$ in \eqref{pdef1} and denote the solution by  $(\bB_h^{n+1,k+1}, \xi_{h}^{n+1,k+1})$\;
    let $\bA_h^{n+1,k+1}=\bA_h^{n+1,k} + \bB_h^{n+1,k+1}$, $\psi_{h}^{n+1,k+1}= \psi_{h}^{n+1,k} + \xi_h^{n+1,k+1}$\;
    compute error, say $err=\|\bB_h^{n+1,k+1}\| + \|\xi_{h}^{n+1,k+1}\|$ \;
    $k=k+1$\;
    }
    }
  \caption{Newton method with preconditioner}
\end{algorithm}

\begin{theorem}\label{th:infsup}
Assume that  $\triangle t\lesssim \frac{1}{1+h^{-1}\|\psi_h^{n+1,k}\|_{\infty}+\|\psi_h^{n+1,k}\|_{\infty}^2}$. There exist positive constants $C_b$ and   $\beta_0$ which are independent the mesh size $h$ such that 
\begin{equation}\label{bdd}
|a^{n+1,k+1}(\bB_h, \xi_h;  \tilde \bA_h, \tilde \psi_h)|\le C_b \left(\|\xi_h\|_h+\|\bB_h\|_h \right) \left(\|\tilde\psi_h\|_h+\|\tilde\bA_h\|_h\right),
\end{equation}
\begin{equation}\label{infsup1}
|a^{n+1,k+1}(\bB_h, \xi_h; \bB_h, \xi_h)|\ge \beta_0(\|\xi_h\|_h+\|\bB_h\|_h)^2.
\end{equation}
\end{theorem}

\begin{proof}
By the Cauchy-Schwarz inequality, the boundedness of $a_1^{n+1,k+1}(\bB_h, \xi_h; \tilde \bA, \tilde \psi_h)$ and $a_4^{n+1,k+1}(\bB_h, \xi_h; \tilde \bA, \tilde \psi_h)$ in~\eqref{adef} is obvious. 
By applying the integration by parts,
\begin{equation}\label{integrationskill}
\begin{aligned}
&\left((\frac{i}{\kappa}\nabla \psi_h^{n+1,k}+\bA_h^{n+1,k}\psi_h^{n+1,k})\cdot \bB_h, \tilde \psi_h\right)\\
=& \left(\psi_h^{n+1,k}, (\frac{i}{\kappa}\nabla +\bA_h^{n+1,k})\cdot (\tilde \psi_h \bB_h)\right) 
-\frac{i}{\kappa}\sum_{K\in \mathcal{T}_h}\langle \psi_h^{n+1,k}n, \tilde \psi_h \bB_h \rangle_{\partial K}
\\
=&\left(\psi_h^{n+1,k}, (\frac{i}{\kappa}\nabla +\bA_h^{n+1,k})\tilde \psi_h\cdot  \bB_h \right) 
+ (\psi_h^{n+1,k}, \tilde \psi_h\frac{i}{\kappa}\nabla\cdot  \bB_h)
-\frac{i}{\kappa}\sum_{K\in \mathcal{T}_h}\langle \psi_h^{n+1,k}n, \tilde \psi_h\bB_h\rangle_{\partial K}
\\
=&\left(\psi_h^{n+1,k}, (\frac{i}{\kappa}\nabla +\bA_h^{n+1,k})\tilde \psi_h\cdot  \bB_h \right) 
- \frac{i}{\kappa}(\nabla(\psi_h^{n+1,k}\tilde \psi_h^*),  \bB_h).
\end{aligned}
\end{equation}  
By the inverse inequality,
\begin{equation}\label{nablaest}
|\frac{i}{\kappa}(\nabla(\psi_h^{n+1,k}\tilde \psi_h^*),  \bB_h)|\le C_0h^{-1}\|\psi_h^{n+1,k}\|_\infty \|\tilde \psi_h\|_0\|\bB_h\|_0,
\end{equation}
which leads to the boundedness of $a_2^{n+1,k+1}(\bB_h, \xi_h; \tilde \bA, \tilde \psi_h)$ and $a_3^{n+1,k+1}(\bB_h, \xi_h; \tilde \bA, \tilde \psi_h)$ in \eqref{adef} directly, and completes the proof for the boundedness in \eqref{bdd}.   
It follows \eqref{integrationskill} that
\begin{equation}\label{acoer}
\begin{aligned}
a^{n+1,k+1}(\bB_h, \xi_h; \bB_h, \xi_h)
=&(\frac{1}{\triangle t} -1)\|\xi_h\|_0^2
+\frac{\sigma}{\triangle t}  \|\bB_h\|_0^2
+ \|\nabla \times \bB_h\|_0^2
+ (|\psi_h^{n+1,k}|^2\bB_h, \bB_h)
+ \|\frac{i}{\kappa}\nabla \xi_h+\bA_h^{n+1,k}\xi_h\|_0^2
\\
&
+ (2|\psi_h^{n+1,k}|^2\xi_h, \xi_h) + \left ((\psi_h^{n+1,k})^2\xi_h^*, \xi_h\right) 
- \frac{i}{\kappa}(\nabla(\psi_h^{n+1,k}\xi_h^*),  \bB_h)
- \hbox{Re}\left(\frac{i}{\kappa}(\nabla(\psi_h^{n+1,k}\xi_h^*),  \bB_h)\right)
\\
&+ 2(\psi_h^{n+1,k}\bB_h, \frac{i}{\kappa}\nabla\xi_h+\bA_h^{n+1,k}\xi_h) 
+ 2\hbox{Re}\left(\psi_h^{n+1,k}\bB_h, \frac{i}{\kappa}\nabla\xi_h+\bA_h^{n+1,k}\xi_h\right).
\end{aligned}
\end{equation}
By the Young's inequality and \eqref{nablaest}, there exist the following estimates for the last five terms on the right-hand side of the above equation
\begin{equation}
\begin{aligned}
-|((\psi_h^{n+1,k})^2\xi_h^*,\xi_h)|
&\ge -(|\psi_h^{n+1,k}|^2\xi_h,\xi_h), 
\\
-|\frac{i}{\kappa}(\nabla(\psi_h^{n+1,k}\xi_h^*),  \bB_h)|
&\ge - \frac{\sigma}{4\Delta t}\|\bB_h\|_0^2 - \frac{C_0^2\Delta t h^{-2}\|\psi_h^{n+1,k}\|_\infty^2}{\sigma} \|\xi_h\|_0^2,
\\
-|(\psi_h^{n+1,k}\bB_h, \frac{i}{\kappa}\nabla\xi_h+\bA_h^{n+1,k}\xi_h)|
&\ge  
-2\|\psi_h^{n+1,k}\|_\infty^2\|\bB_h\|_0^2
- \frac{1}{8}\|\frac{i}{\kappa}\nabla\xi_h+\bA_h^{n+1,k}\xi_h\|_0^2.
\end{aligned}
\end{equation} 
A substitution of the above estimates into the equation \eqref{acoer} leads to 
\begin{equation}\label{eq:31}
\begin{aligned}
|a^{n+1,k+1}(\bB_h, \xi_h; \bB_h, \xi_h)|
\ge&(\frac{1}{\triangle t} -1- \frac{2C_0^2\Delta t h^{-2}\|\psi_h^{n+1,k}\|_\infty^2}{\sigma} )\|\xi_h\|_0^2
+ (\frac{\sigma}{2\triangle t} -8\|\psi_h^{n+1,k}\|_\infty^2) \|\bB_h\|_0^2
+ \|\nabla \times \bB_h\|_0^2
\\
&
+ (|\psi_h^{n+1,k}|^2\bB_h, \bB_h)
+ \frac12\|\frac{i}{\kappa}\nabla \xi_h+\bA_h^{n+1,k}\xi_h\|_0^2
+ (|\psi_h^{n+1,k}|^2\xi_h, \xi_h).
\end{aligned}
\end{equation}
Hence for $\triangle t\lesssim \frac{1}{1+h^{-1}\|\psi_h^{n+1,k}\|_{\infty}+\|\psi_h^{n+1,k}\|_{\infty}^2}$ leads to  the desired conclusion. 
\end{proof} 
\begin{remark}
By Assumption \ref{ass:bound}, $\|\psi_h^{n+1,k}\|_\infty$ is also bounded if the Newton iteration is convergent. 
Hence there exists an independent constant $c_0$ of $h$ such that any $\Delta t\le c_0h$  satisfies the assumption in Theorem \ref{th:infsup}. 
\end{remark}
A combination of Theorem \ref{th:infsup} implies the wellposed-ness of  the linearized system  \eqref{linear2v}. It is known that for a symmetric bilinear form, the wellposed-ness with respect to a norm implies the efficiency of the block preconditioner induced by this norm for the bilinear form in consideration \cite{mardal2011preconditioning,chen2020robust}. 
This motivates the design of the diagonal preconditioner induced by the norm in \eqref{norm1}. Although the bilinear form  $a^{n+1,k+1}(\bB_h^{n+1,k+1}, \xi_{h}^{n+1,k+1}; \tilde \bA_h, \tilde \psi_h)$ is not symmetric, this proposed preconditioner proves efficient numerically.

\subsection{A modified preconditioner for the linearized system}
Note that the preconditioner $\tilde P^{n+1,k+1}$ in \eqref{pdef1} depends on the solution at previous iteration step.  In this section, we propose a modified preconditioner which is independent of the discrete solutions and only needs to be assembled once.

We start with representing the nonlinear system \eqref{Eulerh2} in an equivalent formulation by decoupling complex variables in $\bV_h$ into two real variables in the linear finite space 
$$
V_h=\{\tilde \psi_h\in H^1(\Omega, \R): \tilde \psi_h|_K \in P_1(K,\R), \  \tilde \psi_h \ \mbox{ is continuous on any }\ e\in \cE_h\}.
$$ 
We add the subscript ``$r$" and ``$i$" to the notation of any complex function $  \psi_h\in \bV_h$ to represent its real part and imaginary part, respectively, and denote
$$
u_h^n=(\bA_h^n, \psi_{r,h}^n, \psi_{i,h}^n),\quad \mbox{with }\quad \psi_h^n=  \psi_{r,h}^n + i \psi_{i,h}^n,\qquad \forall 0\le n\le N.
$$  
Given the approximation solution $u_h^n=(\bA_h^n, \psi_{r,h}^n, \psi_{i,h}^n)$ at previous time step $t_n$, the nonlinear system \eqref{Eulerh2} seeks $u_h^{n+1}=(\bA_h^{n+1}, \psi_{r,h}^{n+1}, \psi_{i,h}^{n+1})\in Q_h\times V_h\times V_h$ such that for any $\tilde v_h=(\tilde \bA_h, \tilde \psi_{r,h}, \tilde \psi_{i,h})\in~Q_h\times V_h\times V_h$, 
\begin{equation}\label{newtoneq} 
G_r(u_h^{n+1}; \tilde \psi_{r,h}) + G_i(u_h^{n+1}; \tilde \psi_{i,h}) + G_A(u_h^{n+1}; \tilde \bA_{h})
=
{1\over \Delta t}\left(
(\psi_{r,h}^{n}, \tilde \psi_{r,h}) 
+ (\psi_{i,h}^{n}, \tilde \psi_{i,h}) 
+ \sigma(\bA^{n}_h, \tilde \bA_h) 
\right)
+ \langle \bH,  \boldsymbol{n}\times\tilde{\boldsymbol{A}_h}\rangle_{\partial \Omega},
\end{equation}
where  
\begin{equation} 
\begin{aligned}
G_r(u_h^{n+1}; \tilde \psi_{r,h})=&{1\over \Delta t}(\psi_{r,h}^{n+1}, \tilde \psi_{r,h}) 
+ {1\over \kappa^2}(\nabla \psi_{r,h}^{n+1},\nabla \tilde \psi_{r,h}) 
+ (\bA_h^{n+1} \psi_{r,h}^{n+1}, \bA_h^{n+1} \tilde \psi_{r,h}) 
+ ((|\psi_{r,h}^{n+1}|^2 + |\psi_{i,h}^{n+1}|^2-1)\psi_{r,h}^{n+1}, \tilde \psi_{r,h})  
\\
&-{1\over \kappa}(\nabla \psi_{i,h}^{n+1}, \bA_h^{n+1} \tilde \psi_{r,h}) 
+ {1\over \kappa}( \bA_h^{n+1}\psi_{i,h}^{n+1}, \nabla \tilde \psi_{r,h}),
\\
G_i(u_h^{n+1}; \tilde \psi_{i,h})=&{1\over \Delta t}(\psi_{i,h}^{n+1}, \tilde \psi_{i,h}) 
+ {1\over \kappa^2}(\nabla \psi_{i,h}^{n+1},\nabla \tilde \psi_{i,h})  
+ (\bA_h^{n+1}\psi_{i,h}^{n+1}, \bA_h^{n+1} \tilde \psi_{i,h})
+ ((|\psi_{r,h}^{n+1}|^2 + |\psi_{i,h}^{n+1}|^2-1)\psi_{i,h}^{n+1}, \tilde \psi_{i,h})  
\\
&- {1\over \kappa}(\bA_h^{n+1}\psi_{r,h}^{n+1}, \nabla \tilde \psi_{i,h}) 
+ {1\over \kappa}(\nabla \psi_{r,h}^{n+1}, \bA_h^{n+1} \tilde \psi_{i,h}),
\\
G_A(u_h^{n+1}; \tilde \bA_{h})=&{\sigma\over \Delta t}(\bA^{n+1}_h, \tilde \bA_h) 
+ (\nabla \times \bA_h^{n+1},\nabla \times \tilde \bA_h) 
+ ((|\psi_{r,h}^{n+1}|^2 + |\psi_{i,h}^{n+1}|^2) \bA_h^{n+1}, \tilde \bA_h) 
\\
&- {1\over \kappa}(\nabla \psi_{i,h}^{n+1}, \tilde \bA_h  \psi_{r,h}^{n+1}) 
+ {1\over \kappa}( \nabla \psi_{r,h}^{n+1}, \tilde \bA_h \psi_{i,h}^{n+1}).
\end{aligned} 
\end{equation}
To solve $u_h^{n+1}$ of \eqref{newtoneq}, the Newton method generates a sequence of feasible iterates $u_h^{n+1,k}$ in the form of
\begin{equation}
u_h^{n+1,k+1} = u_h^{n+1,k} + w_h^{n+1,k+1},
\end{equation}
where the initial guess $u_h^{n+1,0}=u_h^{n}$ and the increasment $w_h^{n+1,k+1}=(\bB_h^{n+1,k+1}, \xi_{r,h}^{n+1,k+1}, \xi_{i,h}^{n+1,k+1})$ is the solution to the linearized problem of the nonlinear system \eqref{newtoneq}.
A direct calculation derives the following linearized system of \eqref{newtoneq}
\begin{equation} \label{lineareqform}
a_L^{n+1,k+1}(\bB_h^{n+1,k+1}, \xi_{r,h}^{n+1,k+1}, \xi_{i,h}^{n+1,k+1}; \tilde \bA_h, \tilde \psi_{r,h}, \tilde \psi_{i,h})=f_L^{n+1,k}(\tilde \bA_h, \tilde \psi_{r,h}, \tilde \psi_{i,h}),
\end{equation}
where
$
f_L^{n+1,k}(\tilde \bA_h, \tilde \psi_{r,h}, \tilde \psi_{i,h}) ={\left(
(\psi_{r,h}^{n}, \tilde \psi_{r,h}) 
+(\psi_{i,h}^{n}, \tilde \psi_{i,h})
+ \sigma(\bA^{n}_h, \tilde \bA_h) 
\right)\over \Delta t}
- \left(
G_r(u_h^{n+1,k}; \tilde \psi_{r,h})
+ G_i(u_h^{n+1,k}; \tilde \psi_{i,h})
+ G_A(u_h^{n+1,k}; \tilde \bA_{h})
\right)
$
and
$
a_L^{n+1,k+1}(w_h; \tilde v_h)=\sum_{j=1}^3 a_{L,j}^{n+1,k+1}(w_h; \tilde v_h)
$
with 
\begin{equation} 
\begin{aligned}
a_{L,1}^{n+1,k+1}(w_h; \tilde v_h)
&={1\over \Delta t}(\xi_{r, h},\tilde \psi_{r,h}) 
+ \frac{1}{\kappa^2}(\nabla\xi_{r, h},\nabla \tilde \psi_{r,h}) 
+ \left((|\bA_h^{n+1,k}|^2 + 3|\psi_{r,h}^{n+1,k}|^2 + |\psi_{i,h}^{n+1,k}|^2 -1)\xi_{r, h}, \tilde \psi_{r,h}\right)
\\
& 
-\frac{1}{\kappa}(\bA_h^{n+1,k} \cdot\nabla\xi_{i, h}, \tilde \psi_{r,h})
+\frac{1}{\kappa}(\bA_h^{n+1,k} \xi_{i, h}, \nabla\tilde \psi_{r,h})
+ (2\psi_{r,h}^{n+1,k}\psi_{i,h}^{n+1,k} \xi_{i, h},\tilde \psi_{r,h})
\\
&+ \left((2\psi_{r,h}^{n+1,k}\bA_h^{n+1,k} -\frac{1}{\kappa}\nabla\psi_{i,h}^{n+1,k})\cdot \bB_h, \tilde \psi_{r,h}\right)
+\frac{1}{\kappa}(\psi_{i,h}^{n+1,k}\bB_h, \nabla\tilde \psi_{r,h})
\\
a_{L,2}^{n+1,k+1}(w_h; \tilde v_h)
&={1\over \Delta t}(\xi_{i, h},\tilde \psi_{i,h})
+\frac{1}{\kappa^2}(\nabla\xi_{i, h},\nabla \tilde \psi_{i,h}) 
+ \left((|\bA_h^{n+1,k}|^2 + |\psi_{r,h}^{n+1,k}|^2 + 3|\psi_{i,h}^{n+1,k}|^2 -1)\xi_{i, h}, \tilde \psi_{i,h}\right) 
\\
&
+\frac{1}{\kappa}(\bA_h^{n+1,k}\cdot \nabla\xi_{r, h},  \tilde \psi_{i,h})
-\frac{1}{\kappa}(\bA_h^{n+1,k} \xi_{r, h}, \nabla\tilde \psi_{i,h})
+ 2(\psi_{r,h}^{n+1,k}\psi_{i,h}^{n+1,k}\xi_{r, h} ,\tilde \psi_{i,h})
\\
&+ \left((2\psi_{i,h}^{n+1,k}\bA_h^{n+1,k} +\frac{1}{\kappa}\nabla\psi_{r,h}^{n+1,k})\cdot \bB_h, \tilde \psi_{i,h}\right) 
-\frac{1}{\kappa}(\psi_{r,h}^{n+1,k} \bB_h  , \nabla\tilde \psi_{i,h}),
\\
a_{L,3}^{n+1,k+1}(w_h; \tilde v_h)
&={\sigma\over \Delta t}(\bB_h, \tilde \bA_h)
+(\nabla\times \bB_h,\nabla\times \tilde \bA_h)
+ \left((|\psi_{r,h}^{n+1,k}|^2 + |\psi_{i,h}^{n+1,k}|^2) \bB_h, \tilde \bA_h\right)
\\
&
+ \left((2\psi_{r,h}^{n+1,k}\bA_h^{n+1,k} -\frac{1}{\kappa}\nabla\psi_{i,h}^{n+1,k})\xi_{r, h},\tilde \bA_h\right) 
+ \frac{1}{\kappa}( \psi_{i,h}^{n+1,k}\nabla\xi_{r, h},\tilde \bA_h) 
\\
&
+ \left((2\psi_{i,h}^{n+1,k}\bA_h^{n+1,k} + \frac{1}{\kappa}\nabla \psi_{r,h}^{n+1,k})\xi_{i, h},\tilde \bA_h\right)
-\frac{1}{\kappa}(\psi_{r,h}^{n+1,k}\nabla\xi_{i, h}, \tilde \bA_h) .
\end{aligned} 
\end{equation}
Define a bilinear form
\begin{equation}\label{aP2}
\begin{aligned}
a_{\rm P}(w_h; \tilde v_h)=&{1\over \Delta t}(\xi_{r,h},\tilde \psi_{r,h}) 
+ \frac{1}{\kappa^2}(\nabla\xi_{r,h},\nabla \tilde \psi_{r,h})
+{1\over \Delta t}(\xi_{i,h},\tilde \psi_{i,h})
+\frac{1}{\kappa^2}(\nabla\xi_{i,h},\nabla \tilde \psi_{i,h})
+{\sigma\over \Delta t}(\bB_h, \tilde \bA_h) 
+(\nabla\times \bB_h,\nabla\times \tilde \bA_h),
\end{aligned} 
\end{equation} 
which is derived from the norms of any $w_h=(\bB_h, \xi_{r,h}, \xi_{i,h})\in Q_h\times V_h\times V_h$ defined by
\begin{equation}\label{norm2}
\begin{aligned}
|\|w_h|\|^2=\frac{1}{\triangle t} \|\xi_{r,h}\|_0^2+ \frac{1}{\kappa^2}\|\nabla\xi_{r,h}\|_0^2 
+\frac{1}{\triangle t} \|\xi_{i,h}\|_0^2+ \frac{1}{\kappa^2}\|\nabla\xi_h\|_0^2 
+\frac{\sigma}{\triangle t} \|\bB_h\|_0^2
+ \|\nabla \times\bB_h\|_0^2.
\end{aligned}
\end{equation} 
Let $P$ be the matrix form of the bilinear form defined in \eqref{aP2}
\begin{equation}\label{pdef2}
P=\begin{pmatrix}
P_{\xi_r}&&\\
&P_{\xi_i}&\\
&&P_{\bB}
\end{pmatrix},
\end{equation} 
where $P_{\xi_r}$, $P_{\xi_i}$ and $P_{\bB}$ are the matrix form of the terms in \eqref{aP2} with respect to the variables $\xi_{r,h}$, $\xi_{i,h}$ and $\bB_h$, respectively. The Newton method with the modified preconditioner \eqref{pdef2} is presented in Algorithm \ref{alg:Newton2}. 

\begin{algorithm}[H]
  \SetAlgoLined
  \KwData{Given the initial data $\psi_0$, $\bA_0$, boundary data $\bH$, time step $\Delta t={T\over N}$}
  \KwResult{Approximation $(\bA_h^N, \psi_h^N)$ at time $T=t_N$ }

  Initialization: $(\bA^{0}_h,\psi^{0}_{r,h},\psi^{0}_{i,h})$ are the projections of $\bA_0$, $\mbox{Re}\psi_0$ and $\mbox{Im}\psi_0$ into $Q_h$, $V_h$ and $V_h$, respectively;
  
  \For{$n=1:N-1$}{
  $k=0$, $(\bA^{n+1,0}_h,\psi^{n+1,0}_{r,h},\psi^{n+1,0}_{i,h})=(\bA^{n}_h,\psi^{n}_{r,h}, \psi^{n}_{i,h})$\;
  $err$ is the residual of the nonlinear problem \eqref{newtoneq} \;
  \While{$err>tol$}{
    solve the linear problem \eqref{lineareqform} by GMRES method with preconditioner $  P$ in \eqref{pdef2} and denote the solution by  $w_h^{n+1,k+1}=(\bB_h^{n+1,k+1}, \xi_{r,h}^{n+1,k+1}, \xi_{i,h}^{n+1,k+1})$\;
    let $\bA_h^{n+1,k+1}=\bA_h^{n+1,k} + \bB_h^{n+1,k+1}$, $\psi_{r,h}^{n+1,k+1}= \psi_{r,h}^{n+1,k} + \xi_{r,h}^{n+1,k+1}$, $\psi_{i,h}^{n+1,k+1}= \psi_{i,h}^{n+1,k} + \xi_{i,h}^{n+1,k+1}$\;
    compute error, say $err=|\|w_h^{n+1,k+1}|\| $ \;
    $k=k+1$\;
    }
    }
    
    $\psi_h^N=\psi_{r,h}^N+i\psi_{i,h}^N$\;
  \caption{Newton method with preconditioner}
  \label{alg:Newton2}
\end{algorithm}

A similar result to those in  Theorem \ref{th:infsup} holds for the modified preconditioner as presented in the following theorem.  The boundedness and the coercivity result in Theorem \ref{th:modifiedP} motivates the design of the modified preconditioner derived from the norm in \eqref{norm2}. This proposed preconditioner is employed in this paper and proved to be numerically efficient  in Section \ref{sec:numerical} although the system is non-symmetric.

\begin{theorem}\label{th:modifiedP}
 There is a positive constant $C_L$, which is independent of the mesh size $h$, such that  for any $w_h=(\bB_h, \xi_{r,h}, \xi_{i,h})$ and $\tilde v_h=(\tilde \bA_h, \tilde \psi_{r,h}, \tilde \psi_{i,h})$,
 \begin{equation}\label{ap2}
 a_L^{n+1,k+1}(w_h; \tilde v_h)\le C_L a_{\rm P}(w_h; \tilde v_h).
 \end{equation}
 For $\triangle t\lesssim  \frac{1}{1+ h^{-1}\|\psi_h^{n+1,k}\|_{\infty}+\|\psi_h^{n+1,k}\|^2_\infty+\|\bA_h^{n+1,k}\|^2_\infty}$,
there exists a positive constant $\beta_L$ independent of the mesh size $h$ such that
 \begin{equation}\label{ap3}
 a_L^{n+1,k+1}(w_h; w_h)\ge \beta_L  a_{\rm P}(w_h; w_h).
 \end{equation}
 \end{theorem}
\begin{proof} 
The proof for boundedness of the bilinear form $a_L^{n+1,k+1}(\cdot;\cdot)$ is a direct result of the Cauchy-Schwarz inequality and is omitted here.  The constant $C_L$ in \eqref{ap2} depends on the solution at the previous step of Newton iteration. 

Note that 
\begin{equation} 
\begin{aligned}
&a_{L}^{n+1,k+1}(w_h; w_h)\\
&={1\over \Delta t}(\xi_{r, h}, \xi_{r, h}) 
+ \frac{1}{\kappa^2}(\nabla\xi_{r, h},\nabla \xi_{r, h}) 
+ \left((|\bA_h^{n+1,k}|^2 + 3|\psi_{r,h}^{n+1,k}|^2 + |\psi_{i,h}^{n+1,k}|^2 -1)\xi_{r, h}, \xi_{r, h}\right)
-\frac{2}{\kappa}(\bA_h^{n+1,k} \cdot\nabla\xi_{i, h}, \xi_{r, h})
\\
& 
+\frac{2}{\kappa}(\bA_h^{n+1,k} \xi_{i, h}, \nabla\xi_{r, h})
+ (4\psi_{r,h}^{n+1,k}\psi_{i,h}^{n+1,k} \xi_{i, h},\xi_{r, h})
+ \left((4\psi_{r,h}^{n+1,k}\bA_h^{n+1,k} -\frac{2}{\kappa}\nabla\psi_{i,h}^{n+1,k})\cdot \bB_h, \xi_{r, h}\right)
+\frac{2}{\kappa}(\psi_{i,h}^{n+1,k}\bB_h, \nabla\xi_{r, h})
\\
& 
+{1\over \Delta t}(\xi_{i, h},\xi_{i, h})
+\frac{1}{\kappa^2}(\nabla\xi_{i, h},\nabla\xi_{i, h}) 
+ \left((|\bA_h^{n+1,k}|^2 + |\psi_{r,h}^{n+1,k}|^2 + 3|\psi_{i,h}^{n+1,k}|^2 -1)\xi_{i, h}, \xi_{i, h}\right) 
-\frac{2}{\kappa}(\psi_{r,h}^{n+1,k} \bB_h  , \nabla\xi_{i, h}),
\\
&+ \left((4\psi_{i,h}^{n+1,k}\bA_h^{n+1,k} +2\frac{1}{\kappa}\nabla\psi_{r,h}^{n+1,k})\cdot \bB_h, \xi_{i, h}\right) 
+{\sigma\over \Delta t}(\bB_h,  \bB_h)
+(\nabla\times \bB_h,\nabla\times \bB_h)
+ \left((|\psi_{r,h}^{n+1,k}|^2 + |\psi_{i,h}^{n+1,k}|^2) \bB_h, \bB_h\right).
\end{aligned} 
\end{equation}

By the Young's inequality, we have the following estimates 
\begin{equation}
\begin{aligned}
 (|\bA_h^{n+1,k}|^2 \xi_{r, h}, \xi_{r, h})+(|\psi_{r,h}^{n+1,k}|^2 \bB_h, \bB_h)+(2\psi_{r,h}^{n+1,k}\bA_h^{n+1,k}\cdot \bB_h, \xi_{r, h}) &\ge 0; \\
  (|\bA_h^{n+1,k}|^2 \xi_{i, h}, \xi_{i, h})+(|\psi_{i,h}^{n+1,k}|^2 \bB_h, \bB_h) + (2\psi_{i,h}^{n+1,k}\bA_h^{n+1,k}\cdot \bB_h, \xi_{i, h})&\ge 0; \\
  ( 2|\psi_{r,h}^{n+1,k}|^2 \xi_{r, h}, \xi_{r, h}) + ( 2|\psi_{i,h}^{n+1,k}|^2 \xi_{i, h}, \xi_{i, h}) + (4\psi_{r,h}^{n+1,k}\psi_{i,h}^{n+1,k} \xi_{i, h},\xi_{r, h})&\ge 0.
\end{aligned}
\end{equation}
Hence we have 
\begin{equation} \label{aL2}
\begin{aligned}
&a_{L}^{n+1,k+1}(w_h; w_h)\\
&\ge{1\over \Delta t}\|\xi_{r, h}\|_0^2
+ \frac{1}{\kappa^2}\|\nabla\xi_{r, h}\|^2_0
+ \left(( |\psi_{h}^{n+1,k}|^2 -1)\xi_{r, h}, \xi_{r, h}\right)
+{1\over \Delta t}\|\xi_{i, h}\|_0^2
+\frac{1}{\kappa^2}\|\nabla\xi_{i, h}\|_0^2
+ \left((|\psi_{h}^{n+1,k}|^2 -1)\xi_{i, h}, \xi_{i, h}\right) 
\\
& 
+{\sigma\over \Delta t}\|\bB_h\|_0^2
+\|\nabla\times \bB_h\|_0^2
-\frac{2}{\kappa}(\bA_h^{n+1,k} \cdot\nabla\xi_{i, h}, \xi_{r, h})
+\frac{2}{\kappa}(\bA_h^{n+1,k} \xi_{i, h}, \nabla\xi_{r, h})
+\frac{2}{\kappa}(\psi_{i,h}^{n+1,k}\bB_h, \nabla\xi_{r, h}) 
\\
& 
-\frac{2}{\kappa}(\psi_{r,h}^{n+1,k} \bB_h  , \nabla\xi_{i, h}),
+ \left((2\psi_{i,h}^{n+1,k}\bA_h^{n+1,k} +2\frac{1}{\kappa}\nabla\psi_{r,h}^{n+1,k})\cdot \bB_h, \xi_{i, h}\right) 
+ \left((2\psi_{r,h}^{n+1,k}\bA_h^{n+1,k} -\frac{2}{\kappa}\nabla\psi_{i,h}^{n+1,k})\cdot \bB_h, \xi_{r, h}\right).
\end{aligned} 
\end{equation}
By the Young's inequality, there hold the following estimates for the last six terms in \eqref{aL2}
\begin{equation}\label{aL3}
\begin{aligned}
-\frac{2}{\kappa}(\bA_h^{n+1,k} \cdot\nabla\xi_{i, h}, \xi_{r, h}) &\ge - 2\|\bA_h^{n+1,k}\|^2_\infty\|\xi_{r, h}\|_0^2-\frac{1}{4\kappa^2}\|\nabla\xi_{i, h}\|^2_0; \\
\frac{2}{\kappa}(\bA_h^{n+1,k} \xi_{i, h}, \nabla\xi_{r, h}) &\ge - 2\|\bA_h^{n+1,k}\|^2_\infty\|\xi_{i, h}\|_0^2-\frac{1}{4\kappa^2}\|\nabla\xi_{r, h}\|^2_0;\\
\frac{2}{\kappa}(\psi_{i,h}^{n+1,k}\bB_h, \nabla\xi_{r, h}) &\ge - 2\|\psi_{i,h}^{n+1,k}\|^2_\infty\|\bB_h\|_0^2-\frac{1}{4\kappa^2}\|\nabla\xi_{r, h}\|^2_0;\\
-\frac{2}{\kappa}(\psi_{r,h}^{n+1,k}\bB_h, \nabla\xi_{i, h}) &\ge - 2\|\psi_{r,h}^{n+1,k}\|^2_\infty\|\bB_h\|_0^2-\frac{1}{4\kappa^2}\|\nabla\xi_{i, h}\|^2_0;\\
 \left((2\psi_{i,h}^{n+1,k}\bA_h^{n+1,k} +2\frac{1}{\kappa}\nabla\psi_{r,h}^{n+1,k})\cdot \bB_h, \xi_{i, h} \right)&\ge 
 - \|\psi_{i,h}^{n+1,k}\bA_h^{n+1,k} +\frac{1}{\kappa}\nabla\psi_{r,h}^{n+1,k}\|_\infty (\|\bB_h\|_0^2+\|\xi_{i, h}\|^2_0) ;\\
  \left((2\psi_{r,h}^{n+1,k}\bA_h^{n+1,k} -2\frac{1}{\kappa}\nabla\psi_{i,h}^{n+1,k})\cdot \bB_h, \xi_{r, h}\right)&\ge - \|\psi_{r,h}^{n+1,k}\bA_h^{n+1,k} -\frac{1}{\kappa}\nabla\psi_{i,h}^{n+1,k}\|_\infty (\|\bB_h\|_0^2+\|\xi_{r, h}\|^2_0).\\
  \end{aligned}
\end{equation} 
A combination of \eqref{aL2} and  \eqref{aL3} yields 
\begin{equation} \label{aL4}
\begin{aligned}
&a_{L}^{n+1,k+1}(w_h; w_h)\\
&\ge\left({1\over \Delta t}-2\|\bA_h^{n+1,k}\|^2_\infty-\|\psi_{r,h}^{n+1,k}\bA_h^{n+1,k} -\frac{1}{\kappa}\nabla\psi_{i,h}^{n+1,k}\|_\infty \right)\|\xi_{r, h}\|_0^2
+ \frac{1}{2\kappa^2}\|\nabla\xi_{r, h}\|^2_0
+ \left(( |\psi_{h}^{n+1,k}|^2 -1)\xi_{r, h}, \xi_{r, h}\right)\\
&
+\left({1\over \Delta t} -2\|\bA_h^{n+1,k}\|^2_\infty - \|\psi_{i,h}^{n+1,k}\bA_h^{n+1,k} +\frac{1}{\kappa}\nabla\psi_{r,h}^{n+1,k}\|_\infty \right)\|\xi_{i, h}\|_0^2
+\frac{1}{2\kappa^2}\|\nabla\xi_{i, h}\|_0^2
+ \left((|\psi_{h}^{n+1,k}|^2 -1)\xi_{i, h}, \xi_{i, h}\right) +\|\nabla\times \bB_h\|_0^2
\\
& 
+\left({\sigma\over \Delta t}- 2\|\psi_{i,h}^{n+1,k}\|^2_\infty- 2\|\psi_{r,h}^{n+1,k}\|^2_\infty- \|\psi_{i,h}^{n+1,k}\bA_h^{n+1,k} +\frac{1}{\kappa}\nabla\psi_{r,h}^{n+1,k}\|_\infty -\|\psi_{r,h}^{n+1,k}\bA_h^{n+1,k} -\frac{1}{\kappa}\nabla\psi_{i,h}^{n+1,k}\|_\infty \right)\|\bB_h\|_0^2.
\end{aligned} 
\end{equation}
It follows from $\|\psi_{r,h}^{n+1,k}\bA_h^{n+1,k} -\frac{1}{\kappa}\nabla\psi_{i,h}^{n+1,k}\|_\infty\le \|\psi_{r,h}^{n+1,k}\|_\infty\|\bA_h^{n+1,k}\|_\infty + \frac{1}{\kappa}\|\nabla\psi_{i,h}^{n+1,k}\|_\infty$ and the inverse inequality that the coercivity \eqref{ap3} holds if
\begin{equation}
\triangle t
\lesssim \frac{1}{1+\|\frac{i}{\kappa}\nabla \psi_h^{n+1,k}+\bA_h^{n+1,k}\psi_h^{n+1,k}\|_{\infty}+\|\psi_h^{n+1,k}\|^2_\infty+\|\bA_h^{n+1,k}\|^2_\infty}
\lesssim \frac{1}{1+ h^{-1}\|\psi_h^{n+1,k}\|_{\infty}+\|\psi_h^{n+1,k}\|^2_\infty+\|\bA_h^{n+1,k}\|^2_\infty},
\end{equation}
which completes the proof. 
\end{proof}

Compared to the preconditioner $\tilde P^{n+1,k+1}$ in \eqref{pdef1}, the preconditioner $P$ stays the same for different time step and Newton iteration, so only needs to be assembled once. Although the preconditioner $\tilde P^{n+1,k+1}$ decouples the variable $\xi_h$ and $\bB_h$, the real part and the imaginary part of the complex variable $\xi_h$ is still coupled, while the preconditioner $P$ decouples all the three variables $\xi_{r,h}, \xi_{i,h}$ and $\bB_h$, which leads to an even smaller computational cost.

\section{Numerical Examples}\label{sec:numerical}

In this section, we present some numerical examples on the vortex motion simulations with different geometrics to show the efficiency and robustness of our new scheme and preconditioner under the temporal gauge. The modified preconditioner $P$ in \eqref{pdef2} is employed for all the simulations in this section.

\subsection{Example 1}
Consider the following artificial example on $\Omega=(0,1)^2$ with $\kappa=1$
\begin{equation}\label{artificial} 
\left\{
\begin{aligned} 
\partial_{t} \psi&= -\left(\frac{i}{\kappa} \nabla + \boldsymbol{A}\right)^{2} \psi+\psi-|\psi|^{2} \psi + g& \text { in } \Omega,
\\
\partial_{t} \boldsymbol{A}&=\frac{1}{2 i \kappa }(\psi^* \nabla \psi-\psi \nabla \psi^*)-|\psi|^{2} \boldsymbol{A}-\nabla \times \nabla \times \boldsymbol{A} + f& \text { in } \Omega,
\end{aligned} 
\right.
\end{equation} 
where the boundary conditions are
\begin{equation}\label{artificialbd}  
(\nabla \times \boldsymbol{A})\times \boldsymbol{n}=\boldsymbol{H}_{0}\times \boldsymbol{n},\quad (\frac{i}{\kappa} \nabla + \boldsymbol{A}) \psi\cdot \boldsymbol{n} =0,
\end{equation} 
and initial conditions are
\begin{equation}\label{artificialic}  
\psi(x,0) = \psi_0(x),\quad \boldsymbol{A}(x,0)=\boldsymbol{A}_0(x).
\end{equation} 
The functions $f$, $g$, $\psi_0$ and $\boldsymbol{A}_0$ are chosen corresponding to the exact solution
\[
\psi = e^{-t}(\cos(2\pi x) + i\cos (\pi y)),\quad \boldsymbol{A}=[e^{t-y}\sin (\pi x),\ e^{t-x}\sin(2\pi y)]^T
\]
with 
$
\boldsymbol{H}_{0}= - e^{t-x}\sin (2\pi y) + e^{t-y}\sin (\pi x).
$
We set the terminal time $T=1$ in this example. The initial mesh $\mathcal{T}_1$ consists of two right triangles, obtained by cutting the unit square with a north-east line. Each mesh $\mathcal{T}_i$ is refined into a half-sized mesh uniformly, to get a higher level mesh $\mathcal{T}_{i+1}$.

\begin{table}[htbp]
  \centering 
    \begin{tabular}{c|cc|cc|cc|cc}
\hline
    M      &   $\|\bA-\bA_h\|_{H(\rm curl)}$    & rate&  $\|\psi_r-\psi_{r,h}\|_1$ &rate   &  $\|\psi_i-\psi_{i,h}\|_1$&rate&$\||\psi_h|^2-|\psi|^2\|_0$&rate \\\hline  
2   & 1.38E+00 &      & 1.55E+00 &      & 8.33E-01 &      & 2.71E-01 &      \\\hline
4   & 8.48E-01 & 0.70 & 8.73E-01 & 0.83 & 3.70E-01 & 1.17 & 1.07E-01 & 1.34 \\\hline
8   & 4.39E-01 & 0.95 & 3.17E-01 & 1.46 & 1.32E-01 & 1.49 & 4.56E-02 & 1.23 \\\hline
16  & 2.25E-01 & 0.97 & 1.29E-01 & 1.30 & 6.04E-02 & 1.12 & 1.99E-02 & 1.20 \\\hline
32  & 1.14E-01 & 0.98 & 5.76E-02 & 1.16 & 3.03E-02 & 0.99 & 9.18E-03 & 1.11 \\\hline
64  & 5.72E-02 & 0.99 & 2.72E-02 & 1.08 & 1.53E-02 & 0.98 & 4.40E-03 & 1.06 \\\hline
128 & 2.87E-02 & 1.00 & 1.32E-02 & 1.04 & 7.72E-03 & 0.99 & 2.16E-03 & 1.03 \\\hline
256 & 1.44E-02 & 1.00 & 6.49E-03 & 1.02 & 3.88E-03 & 0.99 & 1.07E-03 & 1.02 \\ \hline
\end{tabular}
\caption{\footnotesize Convergence rate of the nonlinear formulation \eqref{Eulerh2} at $T=1$ with~$\triangle t=1/M$ for Example 1.}
\label{tab:artificial}
\end{table}

We solve the artificial problem \eqref{artificial}   on these uniform triangulations with time step $\triangle t=1/M$, and $M$ is the number of elements on unit length edge. Table~\ref{tab:artificial} lists the errors  at $T=1$. It shows that the convergence rate of $\|\bA-\bA_h\|_{H(curl)}$ and $\|\psi-\psi_h\|_1$ is 1.00. Table~\ref{tab:compare1} compares the average Newton iteration number $N_n$ per time step and the average iteration number for each Krylov iteration on meshes per Newton iteration. Here~$N_{p}$ represents the average Krylov iteration number when the preconditioner in \eqref{pdef2} is employed and $N_{np}$ is the average Krylov iteration number without any preconditioner. Note that the Newton iteration number in Table~\ref{tab:compare1} decreases along as the mesh size.  The reason is that when the mesh is refiner, the discrete solution at the previous time step turns to be a better approximation to the solution at the current time step, namely a better initial guess for the Newton iteration. Thus, only two Newton iteration steps are required for each time step.  

The Krylov iteration number when  no preconditioner is employed increases quickly when the mesh size decreases, which will leads to an unbearable computational cost.
The average Krylov iteration number with the prosed preconditioner in Table~\ref{tab:compare1} does not depend on the mesh size. This behavior indicates the uniform efficiency of the preconditioner and implies a remarkable improvement on the computation speed when it comes to large scale simulations. 

\begin{table}[htbp]
  \centering
    \begin{tabular}{l|ccccccc}
\hline
    M  & 2    & 4    &8    & 16    & 32     & 64    & 128  \\\hline
    $N_n$ &  5.50  & 4.25  & 2.88  & 2.00  & 2.00  & 2.00  & 2.00 \\\hline
    $N_p$ & 18.73 & 15.82 & 11.57 & 9.19  & 7.92  & 6.70  & 6.00 \\\hline
    $N_{np}$ & 88.55 & 254.12 & 207.78 & 309.22 & 431.19 &  -  &  - \\\hline
    \end{tabular}%
  \caption{\footnotesize Comparison of average iterations  with $\triangle t=1/M$ for Example 1.}
  \label{tab:compare1}%
\end{table}

\subsection{Example 2: Unit square superconductor}
We simulate the vortex dynamics \eqref{model0temporal} on a unit square domain $\Omega=(0,1)^2$ with $\kappa=10$ and initial conditions
\begin{equation} 
\psi(x,0) = 0.6+0.8i,\quad \boldsymbol{A}(x,0)=(0,0), \quad H=5.
\end{equation} 
This example was tested before in \cite{chen2001adaptive,yang2014linearized,li2015new,gao2017efficient,mu1997linearized,gao2014optimal}. We triangulate the domain into uniform right triangles with $M$ points on each side, and solve the equations  with the time step size $\triangle t=1/M$. 

\begin{table}[htbp]
  \centering
    \begin{tabular}{l|ccccccc}
\hline
    M  & 2    & 4    &8    & 16    & 32     & 64  \\\hline
    $N_n$ & 1.25  & 1.19  & 1.18  & 1.10  & 1.05  & 1.03 \\\hline
    $N_p$ & 3.62  & 4.84  & 4.78  & 4.11  & 3.03  & 2.41 \\\hline
    $N_{np}$ & 7.78  & 26.83  & 59.59 & 94.89  & 105.50  & 105.54 \\\hline 
    \end{tabular}%
  \caption{\footnotesize Comparison of average iterations  with $\triangle t=1/M$ for Example 2.}
  \label{tab:compare2}%
\end{table}%

Table~\ref{tab:compare2} records the average Newton iteration number $N_n$ per time step and the average Krylov iteration number $N_p$ and $N_{np}$ per Newton step in Example 2.   The comparison of the average iteration numbers in Table~\ref{tab:compare2} verifies the efficiency of the proposed preconditioner, which will significantly speed up large scale simulations.
Figure~\ref{fig:square} plots the value of $|\psi|^2$ and $\nabla\times \bA$ at different time levels on the mesh $M=16$, which is similar to those reported in \cite{li2015new,gao2017efficient}.
\begin{figure}[!ht]
\centering
\subfloat[$|\psi|^2$ at $T=2$, $6$, $10$, $15$ and $20$]{%
\centering
\includegraphics[width = 1.2in]{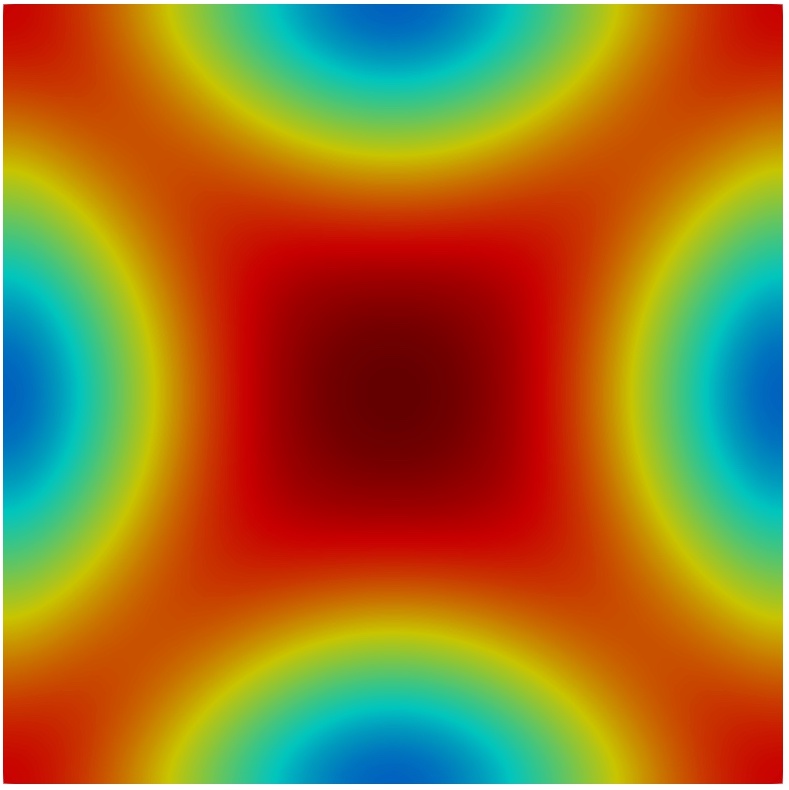}
\includegraphics[width = 1.2in]{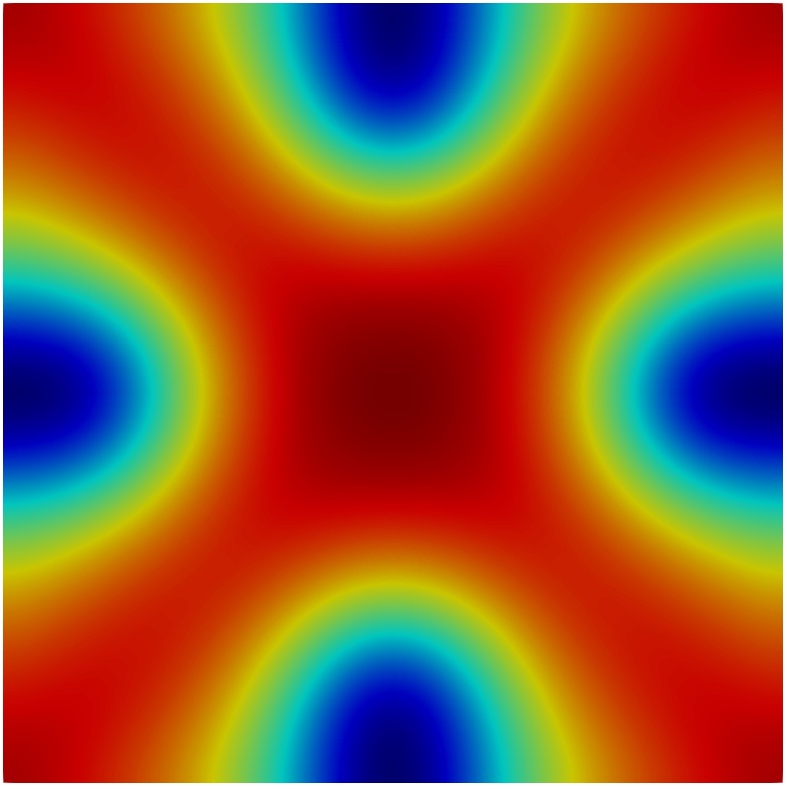}
\includegraphics[width = 1.2in]{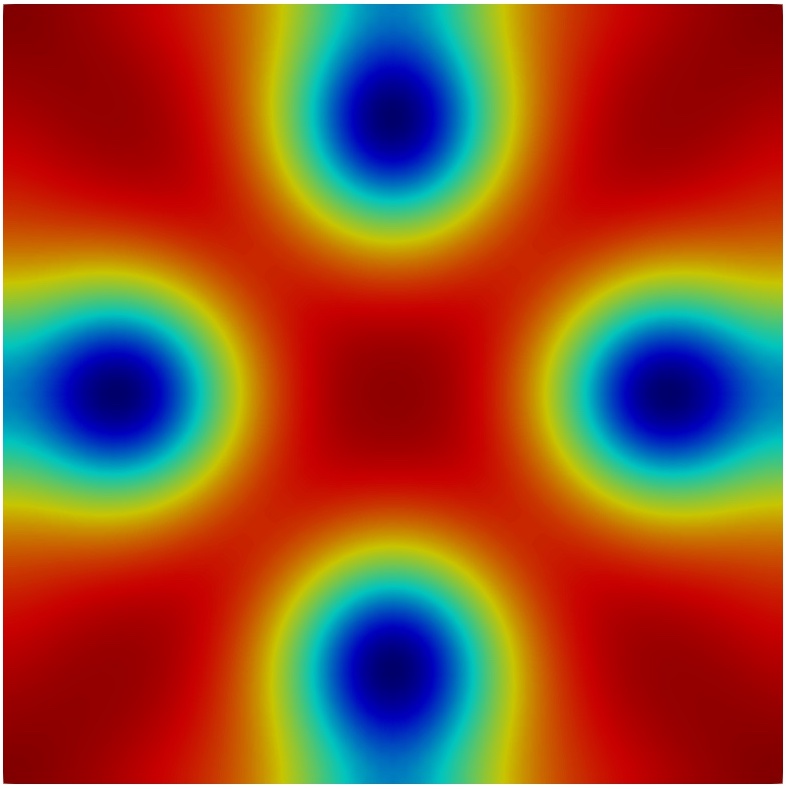}
\includegraphics[width = 1.2in]{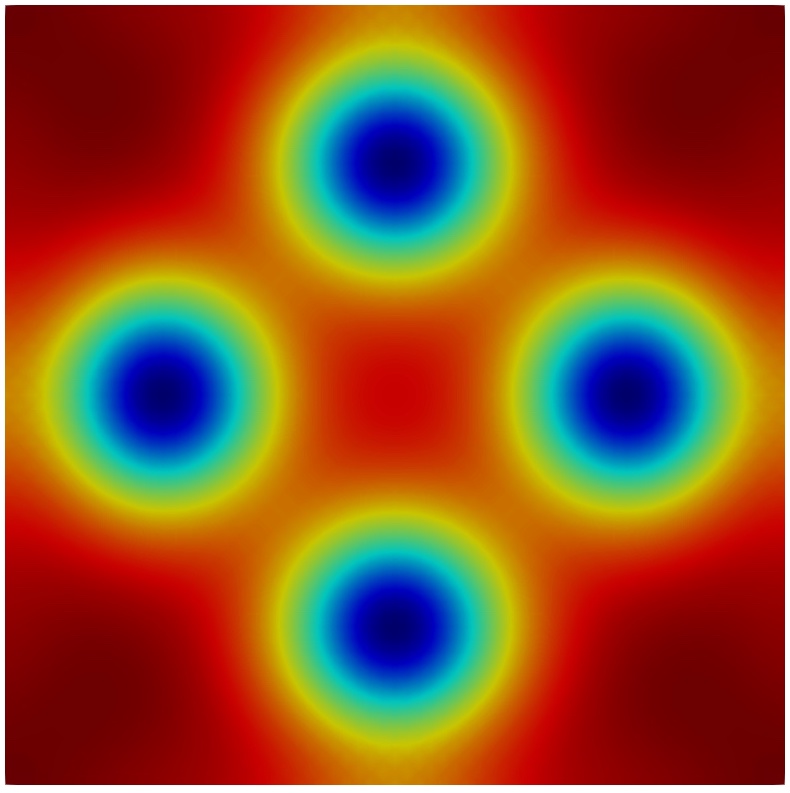}
\includegraphics[width = 1.2in]{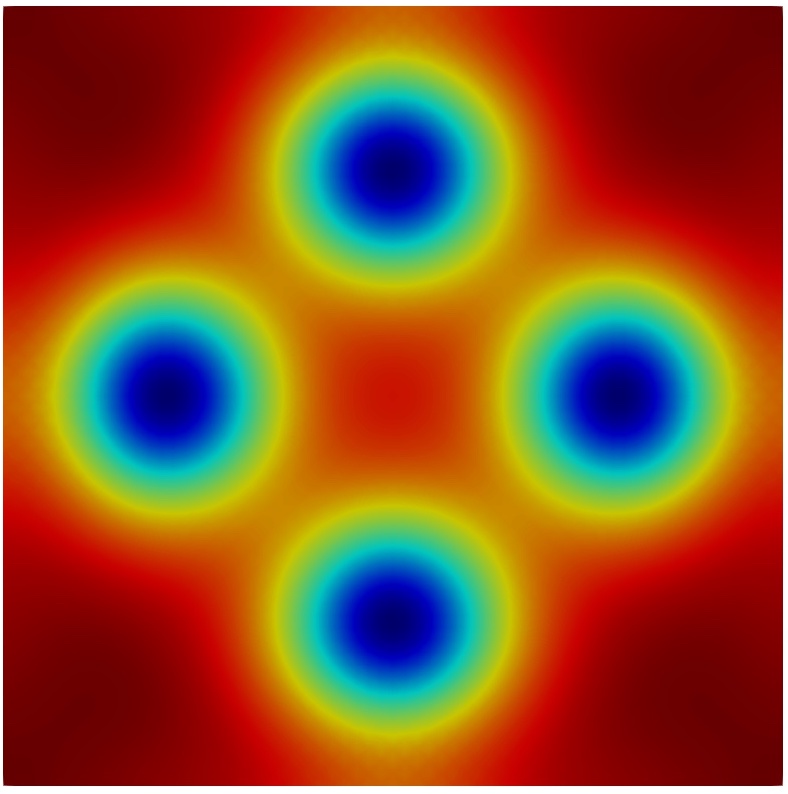}
}

\subfloat[$\nabla\times \bA$ at $T=2$, $6$, $10$, $15$ and $20$]{%
\centering
	\includegraphics[width = 1.2in]{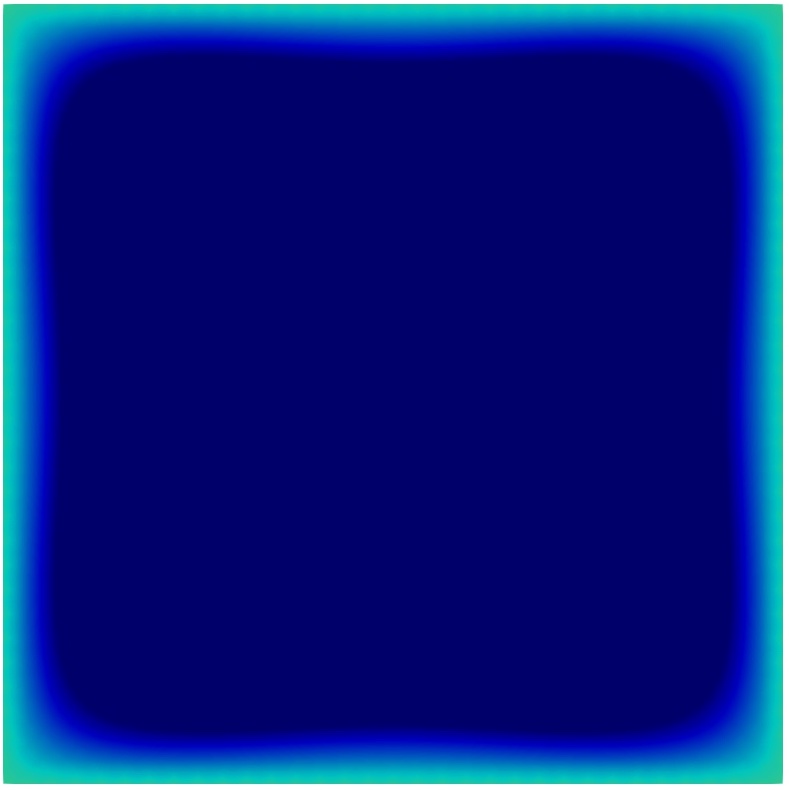}
	\includegraphics[width = 1.2in]{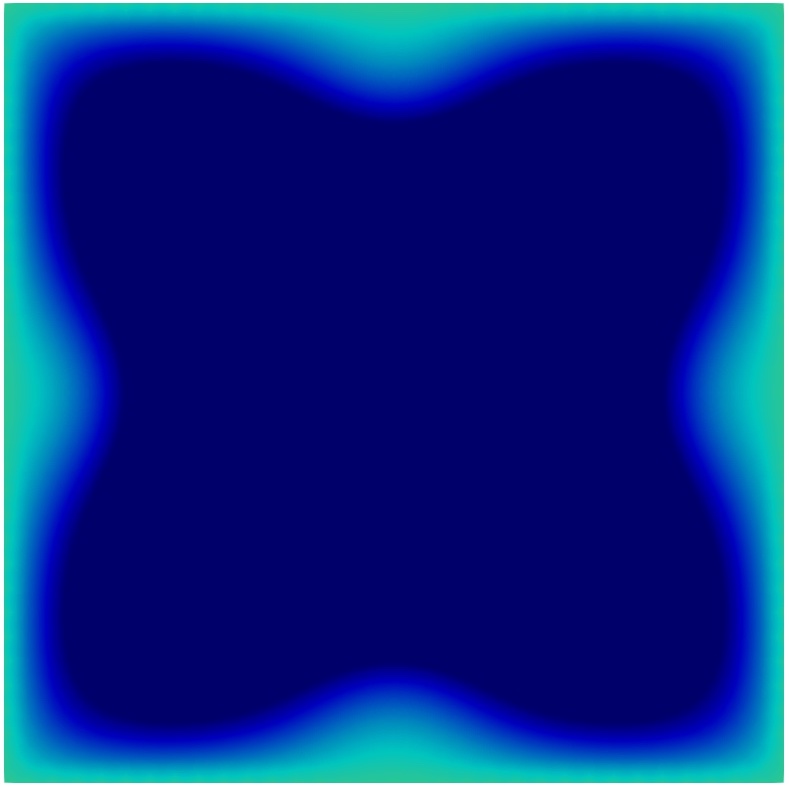}
	\includegraphics[width = 1.2in]{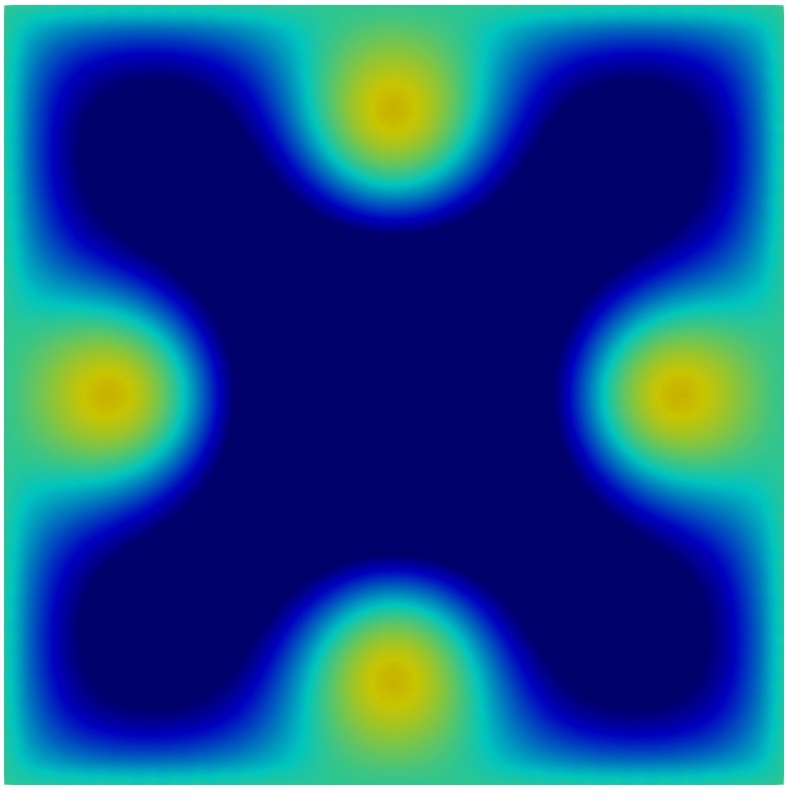}
	\includegraphics[width = 1.2in]{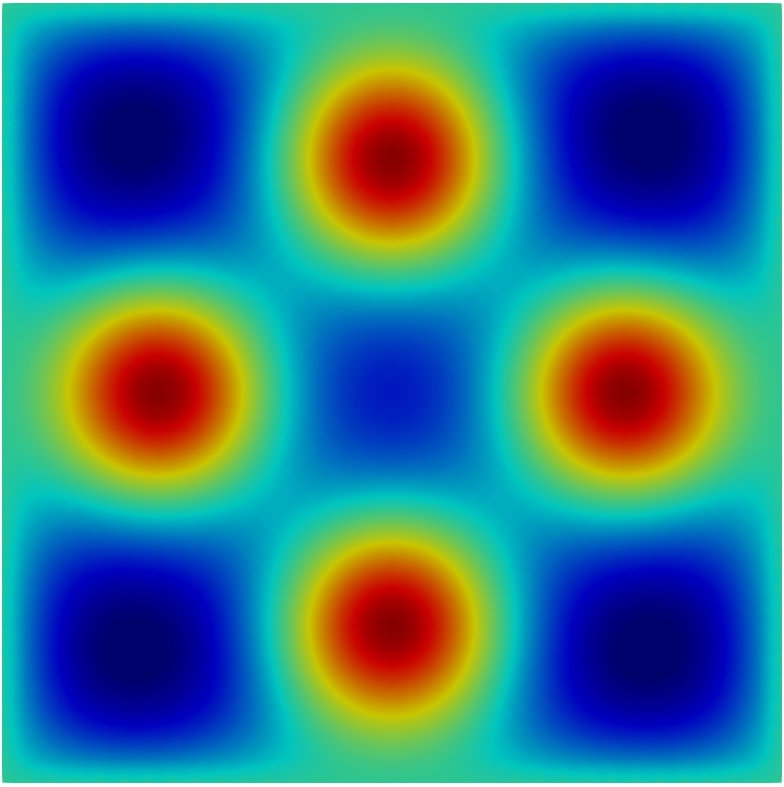}
	\includegraphics[width = 1.2in]{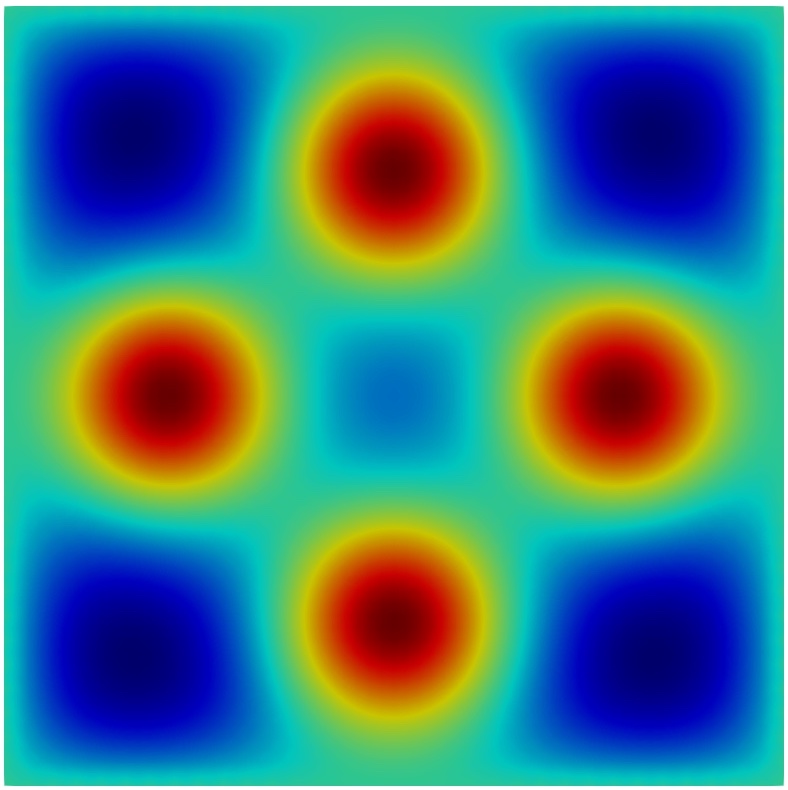}
}

\caption{\footnotesize $|\psi|^2$ and $\nabla\times \bA$ on the unit square domain with $M$=16.}
\label{fig:square}
\end{figure} 

\subsection{Example 3: L-shaped superconductor} 
We use the prosed formulation and preconditioner to simulate the vortex dynamics in an L-shaped superconductor $\Omega=(0,1)^2\backslash [0.5, 1]\times [0,0.5]$ with the Ginzburg-Landau parameter $\kappa=10$. The initial conditions and applied magnetic field are
\begin{equation} 
\psi(x,0) = 0.6+0.8i,\quad \boldsymbol{A}(x,0)=(0,0), \quad H=5.
\end{equation} 
This example was tested before by different methods, see \cite{gao2017efficient,li2015new} for reference. The L-shaped domain is triangulated quasi-uniformly with $M$ nodes per unit length on each side, where Figure \ref{fig:Lshape} plots the case with~$M=16$. 

\begin{figure}[!ht]
\centering
	\includegraphics[width = 1.5in]{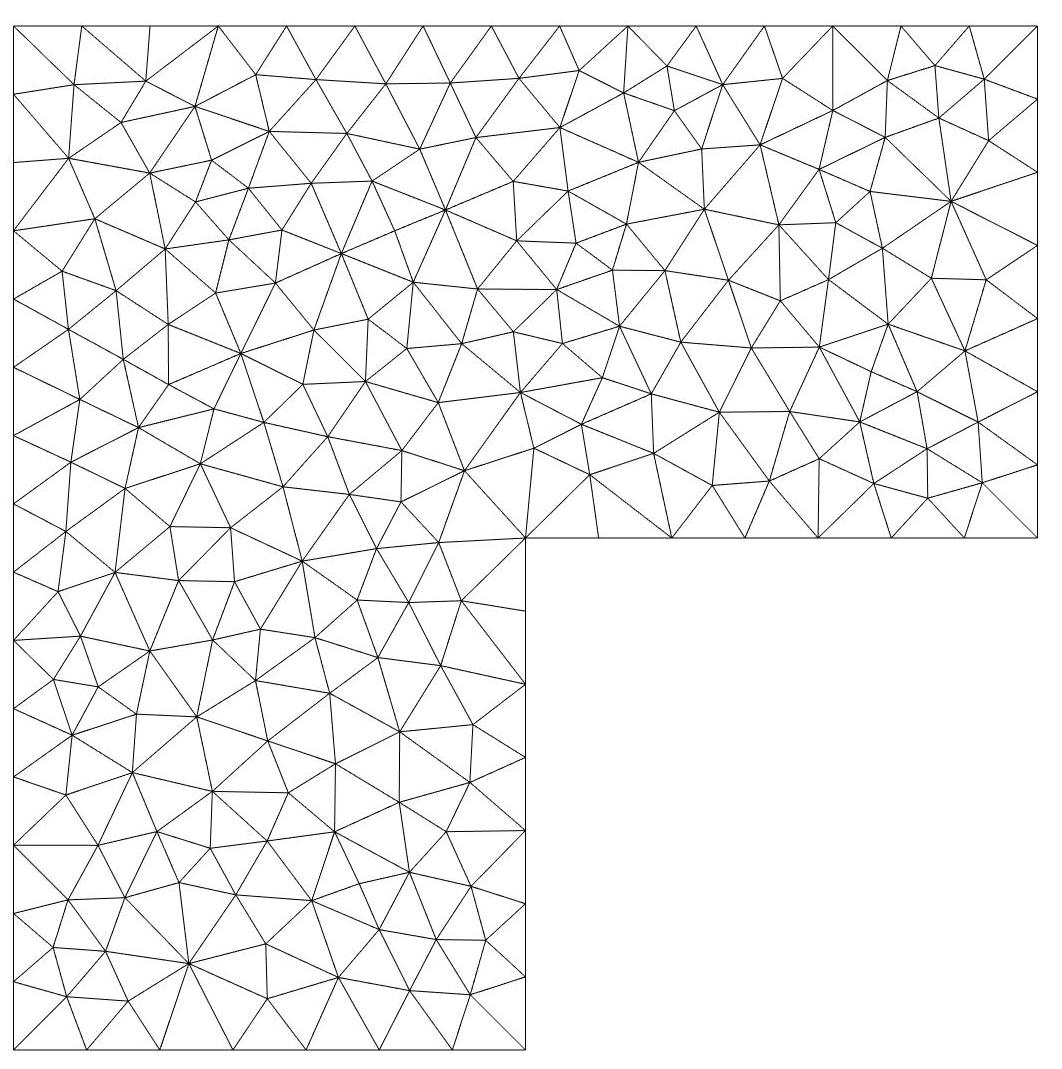}
	\caption{\footnotesize Quasi-uniform triangulation of the L-shape domain with $M$=16.}
	\label{fig:Lshape}
\end{figure}

Table \ref{tab:compare3} records the Newton iteration number per time step and the Krylov iteration number per Newton step,
which implies the uniform efficiency of the proposed preconditioner.

\begin{table}[htbp]
  \centering
    \begin{tabular}{l|cccccc}
\hline
    M  &  4    &8    & 16    & 32     & 64   \\\hline
    $N_n$ & 1.08  & 1.28  & 1.04  & 1.02  & 1.01 \\\hline
    $N_p$ & 3.91  & 5.11  & 3.55  & 2.70  & 1.98\\\hline
    $N_{np}$ & 16.31  & 64.28  & 83.14  & 98.58  & 131.47\\\hline 
    \end{tabular}%
  \caption{\footnotesize Comparison of average iterations  with $\triangle t=1/M$ for Example 3.}
  \label{tab:compare3}%
\end{table}%

\begin{figure}[!ht]
\centering
\subfloat[$|\psi|^2$  at $T=5$, $10$, $25$ and $40$  ]{%
\centering
\includegraphics[width = 1.2in]{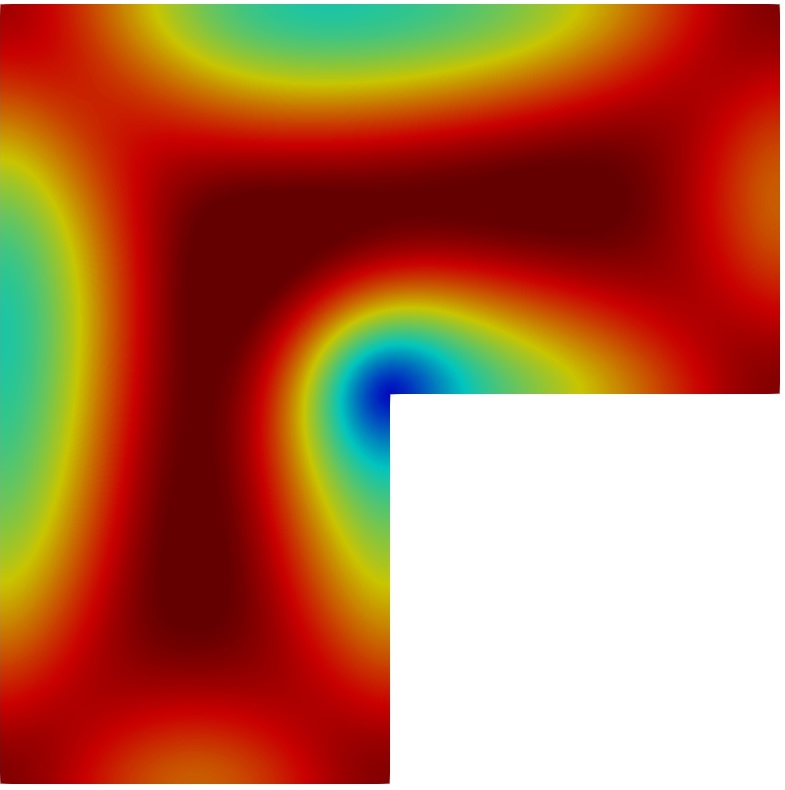}
\includegraphics[width = 1.2in]{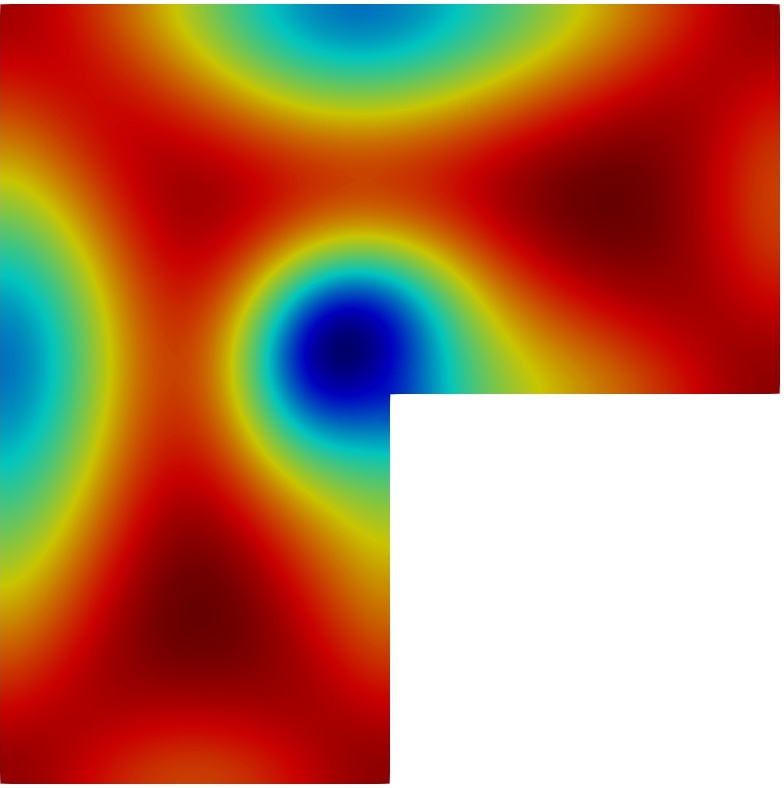}
\includegraphics[width = 1.2in]{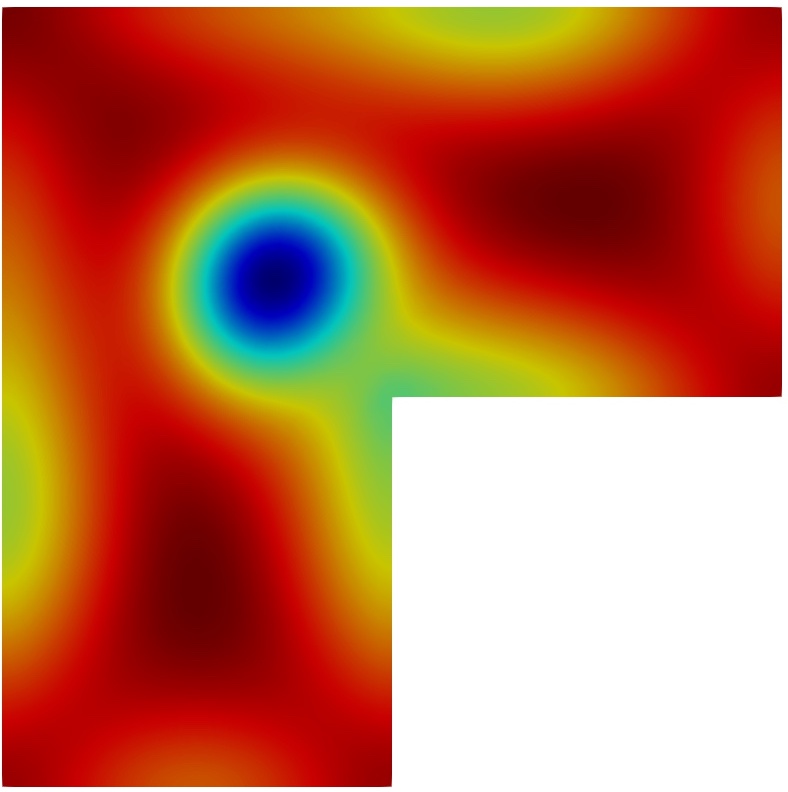}
\includegraphics[width = 1.2in]{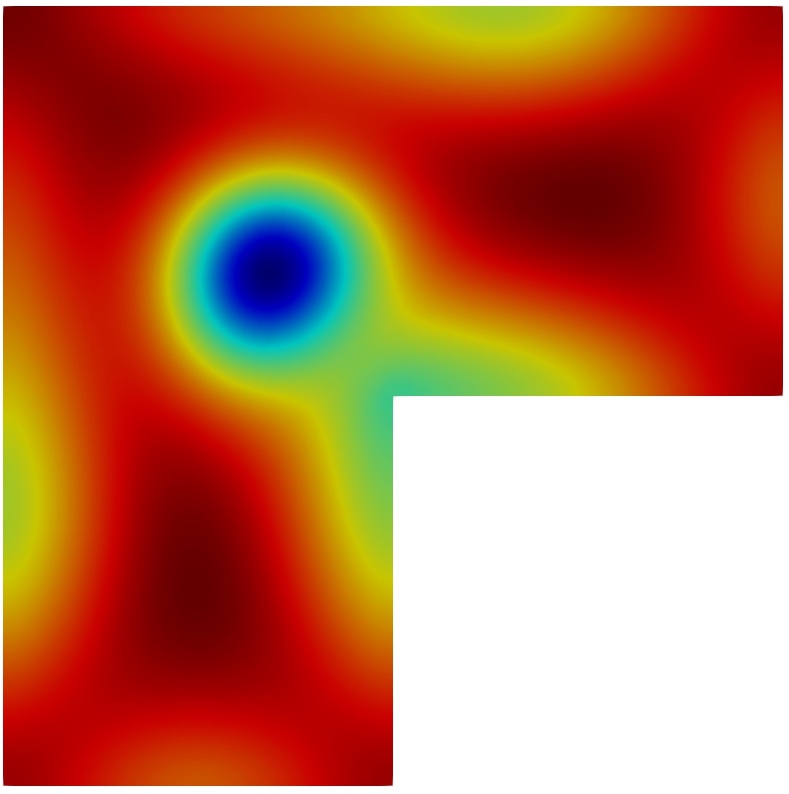} 
}

\subfloat[$\nabla\times \bA$  at $T=5$, $10$, $25$ and $40$  ]{%
\centering
	\includegraphics[width = 1.2in]{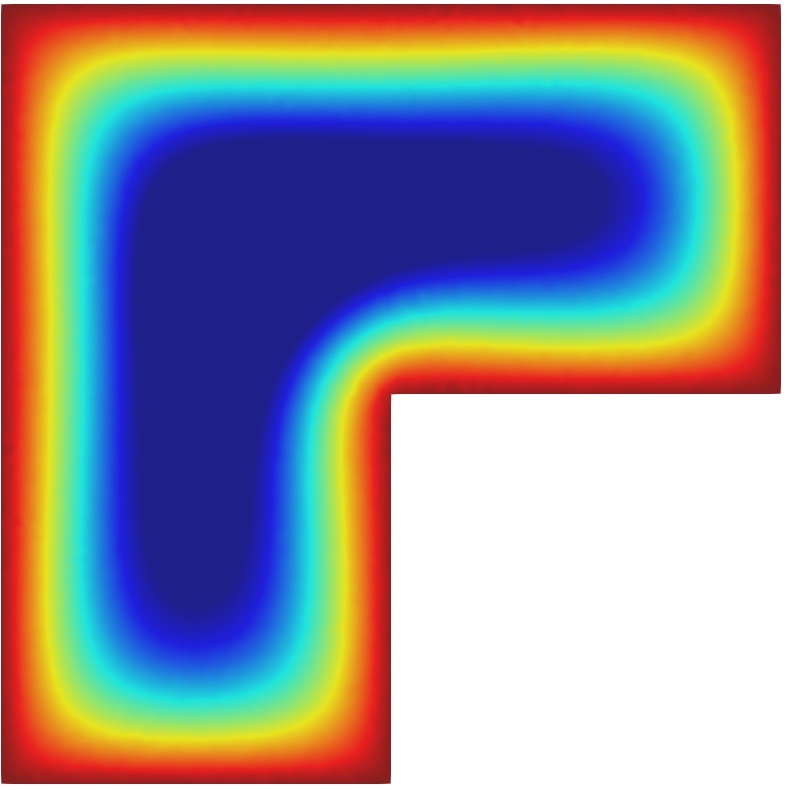}
	\includegraphics[width = 1.2in]{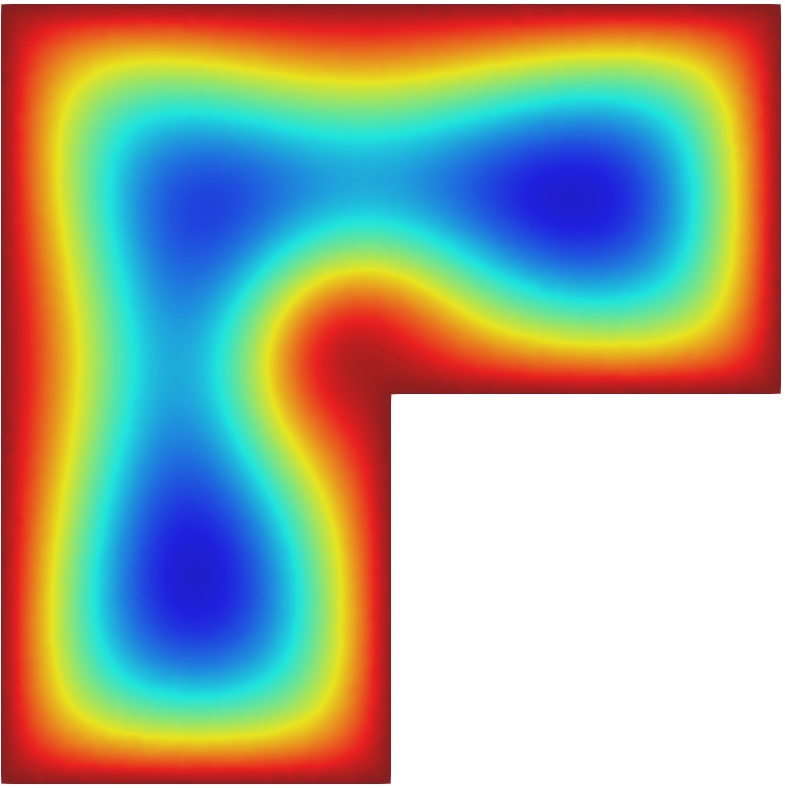}
	\includegraphics[width = 1.2in]{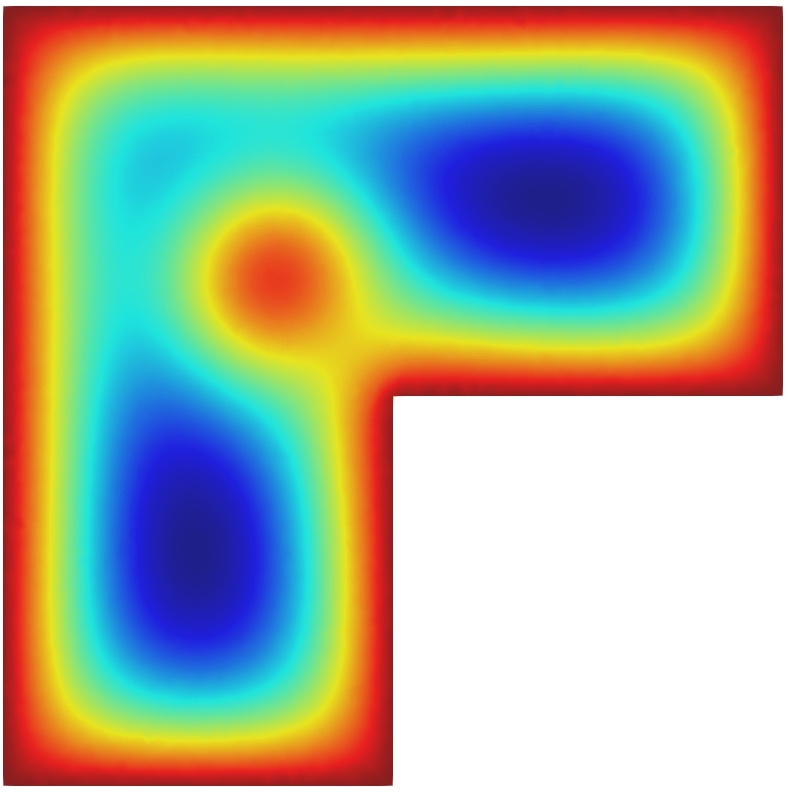}
	\includegraphics[width = 1.2in]{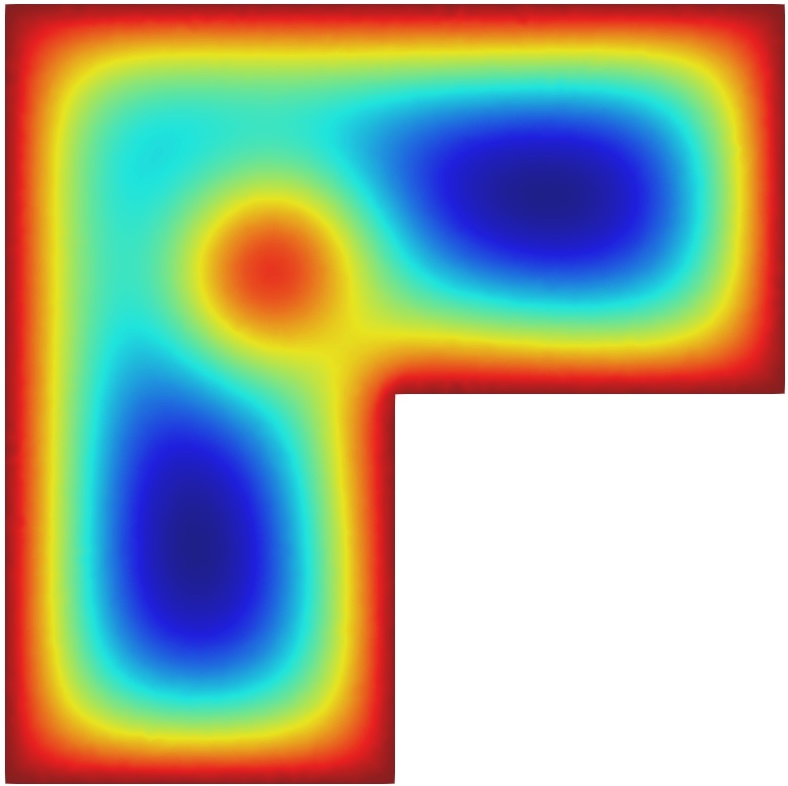} 
}

\caption{\footnotesize $|\psi|^2$ and $\nabla\times \bA$ on L-shaped domain with $M$=16.}
\label{fig:L}
\end{figure} 
Figure \ref{fig:L}  plots the value of $|\psi|^2$ at $T=5$, $20$ and $T=40$ by the new proposed method with $\triangle t=1/M$. As showed in Figure~\ref{fig:L}, one vortex enters the material from the re-entrant corner as the time increases, which is similar to those reported in \cite{gao2017efficient,li2015new}.  It was reported in \cite{gao2017efficient,li2015new} that  the conventional finite element method in $H^1(\Omega)$ for solving the Ginzburg-Landau equations under temporal gauge is unstable with respect to the mesh size. To be specific, this conventional method with $M=16$ and $32$ gives a nonphysical simulation when $T=40$, but the one with $M=64$ exhibits the correct phenomenon.

Figure~\ref{fig:L} shows that the numerical solution of the proposed approach on the mesh $M=16$, and the simulations on the meshes $M=32$ and $M=64$ are similar to those in Figure~\ref{fig:L}. This implies that the new approach is stable and correct. The reason why the proposed approach works while the conventional one does not is that the true solution $\bA$ of this problem is not $H^1$ any more. The conventional finite element solves $\bA$ in a finite dimensional $H^1$ space, thus only gives an approximation to a projection of $\bA$, not an approximation to $\bA$, and leads to the unstable behavior.

\subsection{Example 4} 
We present simulations of vortex dynamics of a type II superconductor in a square domain with four square holes. 
\begin{figure}[!ht]
\begin{center}
\begin{tikzpicture}[xscale=0.3,yscale=0.3]
\draw[-] (0,0) -- (10,0);
\draw[-] (0,0) -- (0,10);
\draw[-] (0,10) -- (10,10);
\draw[-] (10,0) -- (10,10);
\draw[-] (2,2) -- (3,2);
\draw[-] (2,2) -- (2,3);
\draw[-] (2,3) -- (3,3);
\draw[-] (3,2) -- (3,3);

\draw[-] (2,7) -- (3,7);
\draw[-] (2,7) -- (2,8);
\draw[-] (2,8) -- (3,8);
\draw[-] (3,7) -- (3,8);

\draw[-] (7,2) -- (8,2);
\draw[-] (7,2) -- (7,3);
\draw[-] (7,3) -- (8,3);
\draw[-] (8,2) -- (8,3);

\draw[-] (7,7) -- (8,7);
\draw[-] (7,7) -- (7,8);
\draw[-] (7,8) -- (8,8);
\draw[-] (8,7) -- (8,8);

\node[left] at (0,2) {2};
\node[left] at (0,3) {3};
\node[left] at (0,1) {1};
\node[left] at (0,4) {4}; 
\node[left] at (0,5) {5};
\node[left] at (0,6) {6};
\node[left] at (0,7) {7};
\node[left] at (0,8) {8}; 
\node[left] at (0,9) {9}; 
\node[left] at (0,10) {10}; 

\draw[fill] (1,0) circle [radius=0.05];
\draw[fill] (2,0) circle [radius=0.05];
\draw[fill] (3,0) circle [radius=0.05];
\draw[fill] (4,0) circle [radius=0.05];
\draw[fill] (5,0) circle [radius=0.05];
\draw[fill] (6,0) circle [radius=0.05];
\draw[fill] (7,0) circle [radius=0.05];
\draw[fill] (8,0) circle [radius=0.05];
\draw[fill] (9,0) circle [radius=0.05];
\draw[fill] (0,1) circle [radius=0.05];
\draw[fill] (0,2) circle [radius=0.05];
\draw[fill] (0,3) circle [radius=0.05];
\draw[fill] (0,4) circle [radius=0.05];
\draw[fill] (0,5) circle [radius=0.05];
\draw[fill] (0,6) circle [radius=0.05];
\draw[fill] (0,7) circle [radius=0.05];
\draw[fill] (0,8) circle [radius=0.05];
\draw[fill] (0,9) circle [radius=0.05];
\node[below] at (1,0) {1};
\node[below] at (2,0) {2};
\node[below] at (3,0) {3};
\node[below] at (4,0) {4};
\node[below] at (5,0) {5};
\node[below] at (6,0) {6}; 
\node[below] at (7,0) {7};
\node[below] at (8,0) {8};
\node[below] at (9,0) {9};
\node[below] at (10,0) {10}; 
\node[below,left] at (0,0) {0}; 
\end{tikzpicture}
\caption{A square with four holes for Example 4.}
\label{fig:hole}
\end{center}
\end{figure}
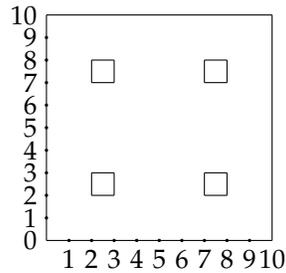

We set 
$$
\sigma=1,\quad  \kappa =4,\quad  \psi(x,0)=1.0,\quad \boldsymbol{A}(x,0)=(0,0),
$$
and test on three different external magnetic fields, namely $\bH=0.8$ and $1.1$.  The example was tested before in \cite{gao2016new,peng2014vortex}. 

\begin{figure}[!ht]
\centering
\subfloat[$|\psi|^2$  at $T=10$, $20$, $50$, $300$ and $500$ ]{%
\centering
\includegraphics[width = 1.2in]{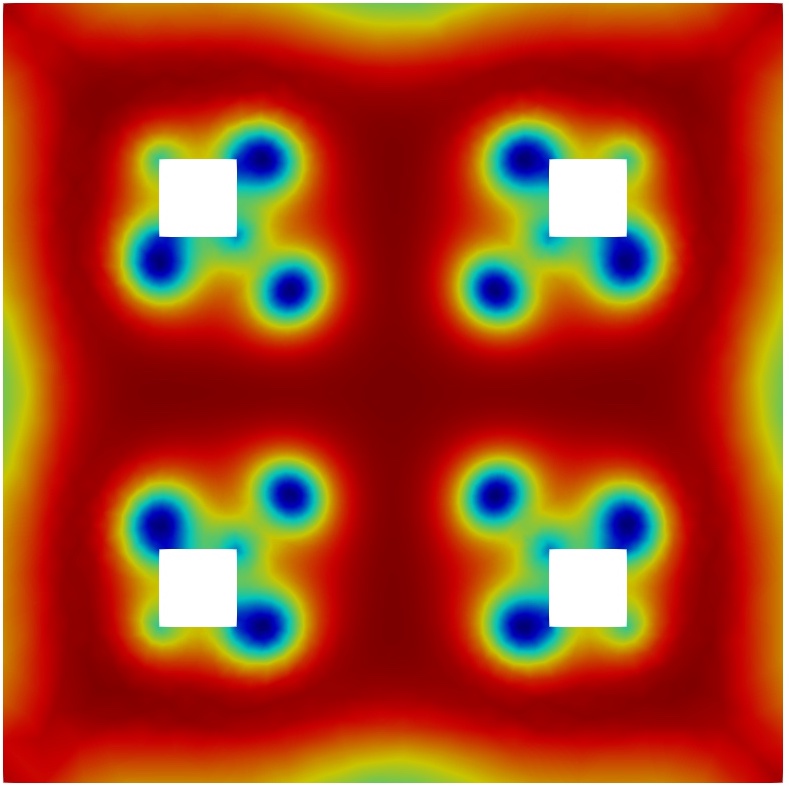}
\includegraphics[width = 1.2in]{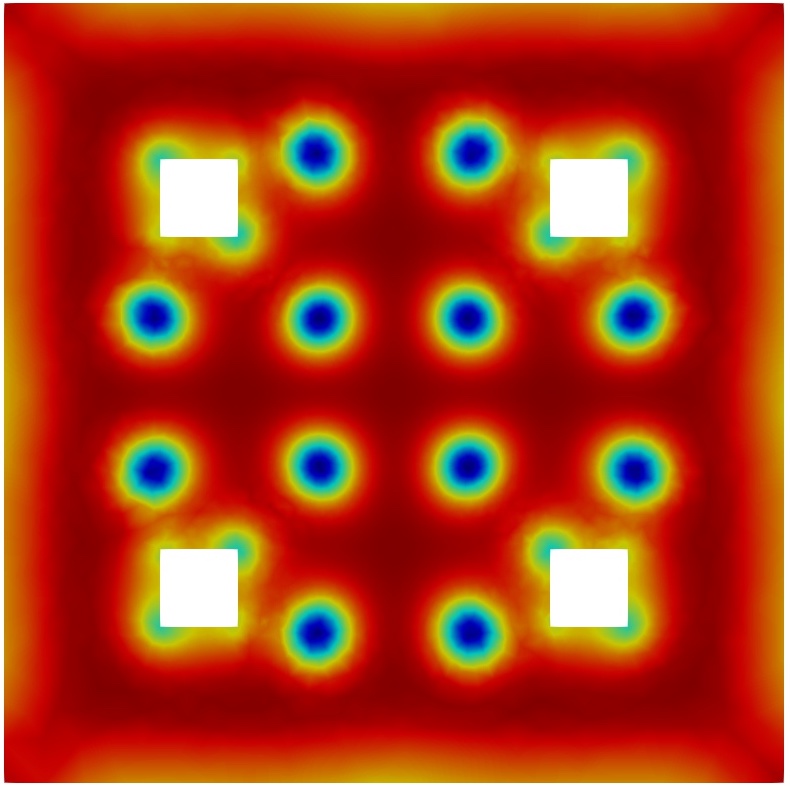}
\includegraphics[width = 1.2in]{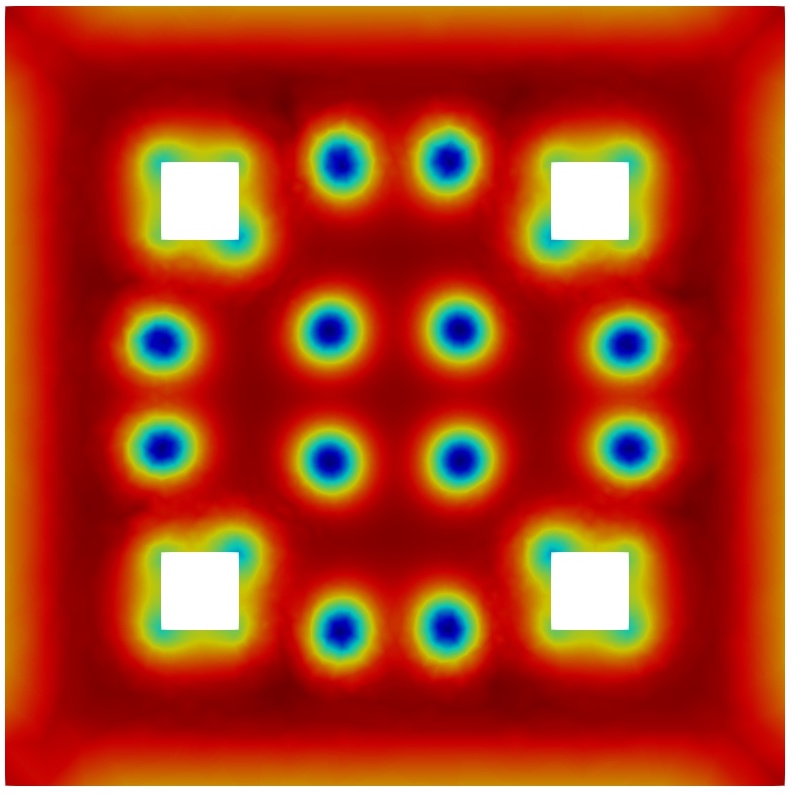}
\includegraphics[width = 1.2in]{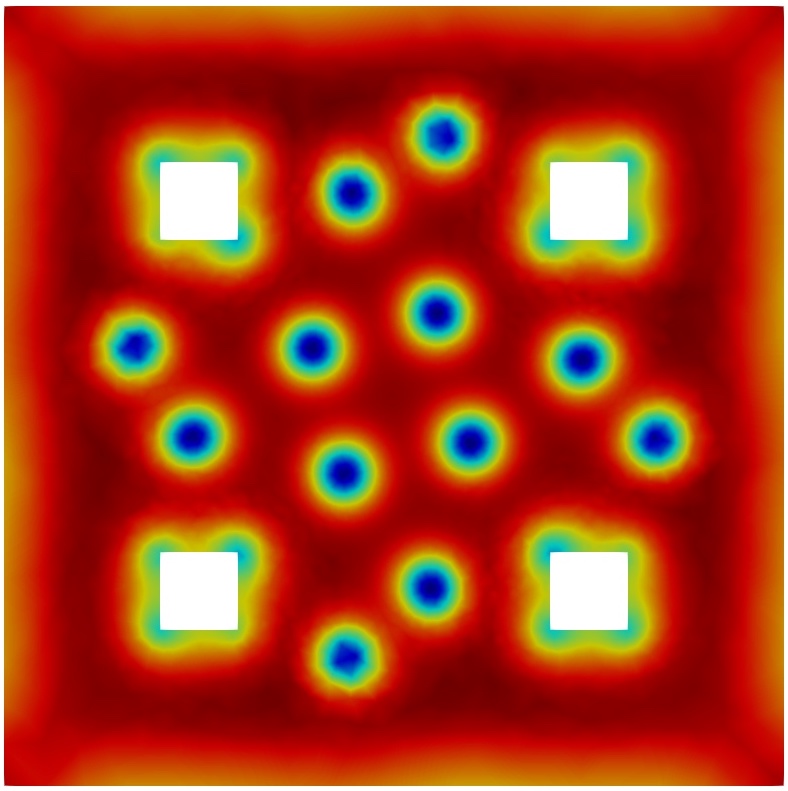} 
\includegraphics[width = 1.2in]{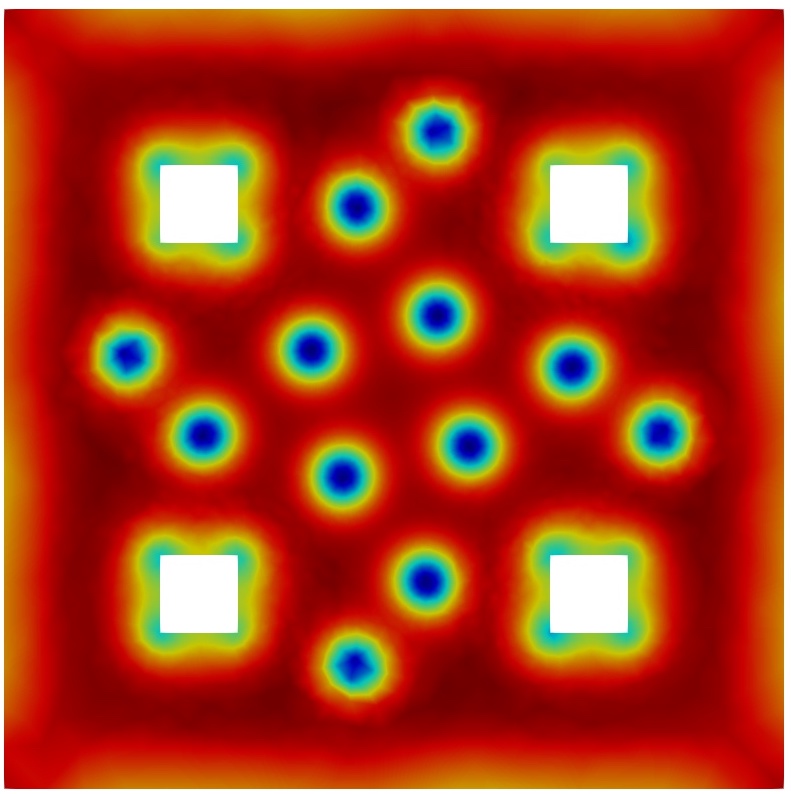} 
}

\subfloat[$\nabla\times \bA$  at $T=10$, $20$, $50$, $300$ and $500$  ]{%
\centering
	\includegraphics[width = 1.2in]{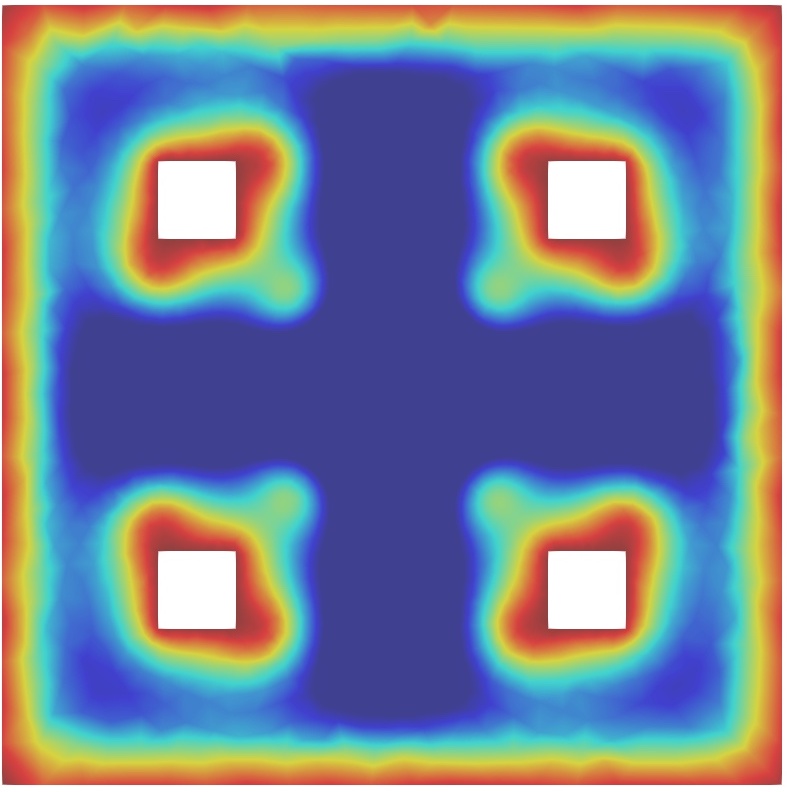}
	\includegraphics[width = 1.2in]{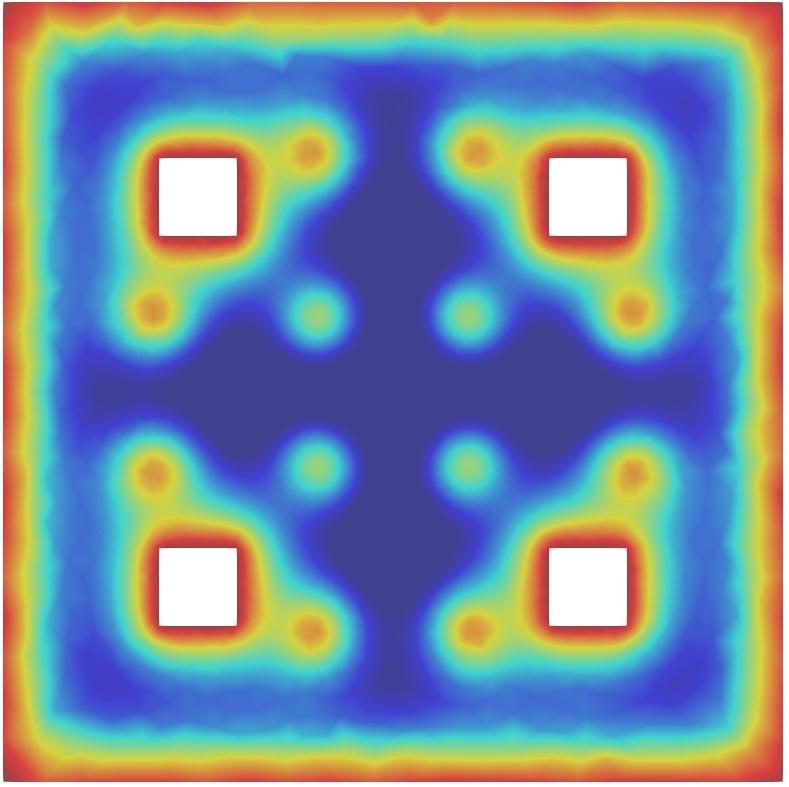}
	\includegraphics[width = 1.2in]{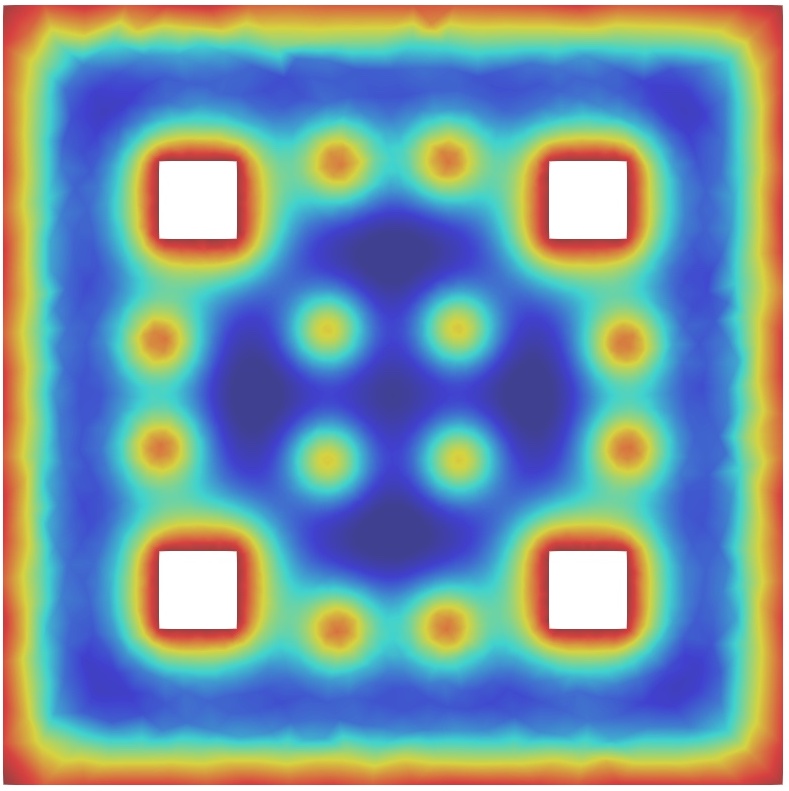}
	\includegraphics[width = 1.2in]{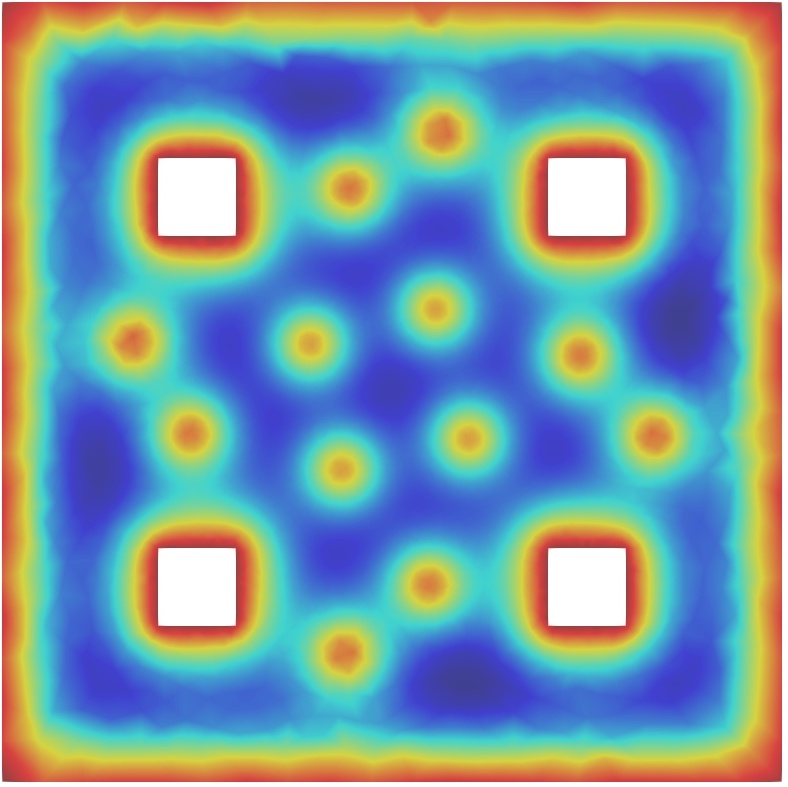} 
	\includegraphics[width = 1.2in]{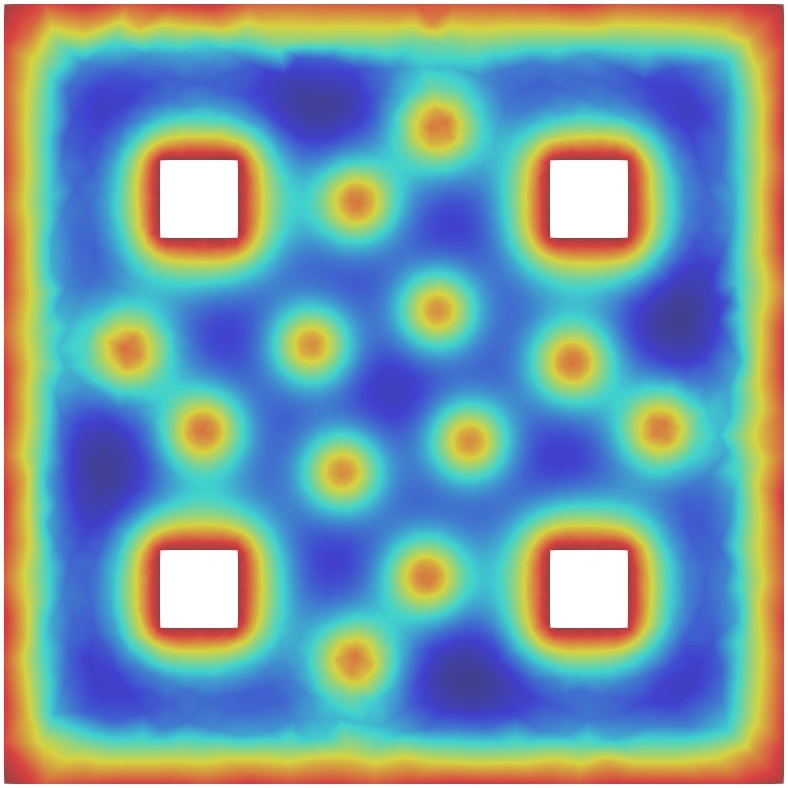} 
}

\caption{\footnotesize $|\psi|^2$ and $\nabla\times \bA$ on the squared domain with four holes for $\bH=0.8$.}
\label{fig:hole08}
\end{figure} 

The simulation for $\bH=0.8$ is conducted on a quasi-uniform mesh with 
8144 elements and the time step $\triangle t=0.02$. Figure~\ref{fig:hole08} plots the value of $|\psi|^2$ and $\nabla\times \bA$ at time $T=10$, $20$, $50$, $300$ and $500$, where the simulation until $T=2000$ shows that the vortex pattern stays unchanged after $T=500$. It shows that the vortices start to penetrate the material near the four square holes. 
Figure~\ref{fig:hole11} plots the simulation for $\bH=1.1$ on a quasi-uniform mesh with 305550 elements with $\triangle t=0.02$.
It clearly shows that more vortices are generated and earlier stationary state as the applied magnetic field $\bH$ increases.

\begin{figure}[!ht]
\centering
\subfloat[$|\psi|^2$  at $T=10$, $20$, $50$, $100$ and  $500$ ]{%
\centering
	\includegraphics[width = 1.2in]{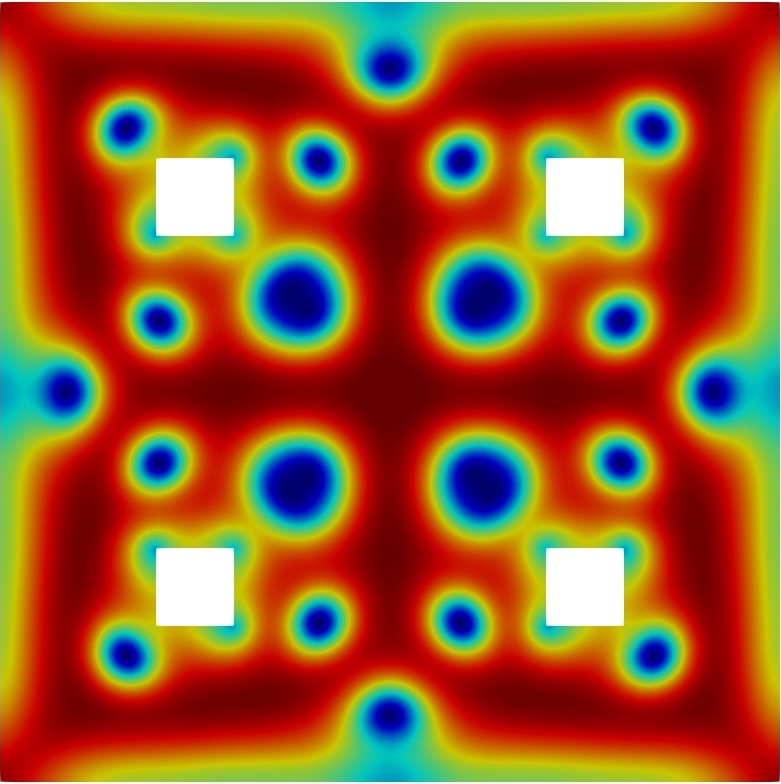}
	\includegraphics[width = 1.2in]{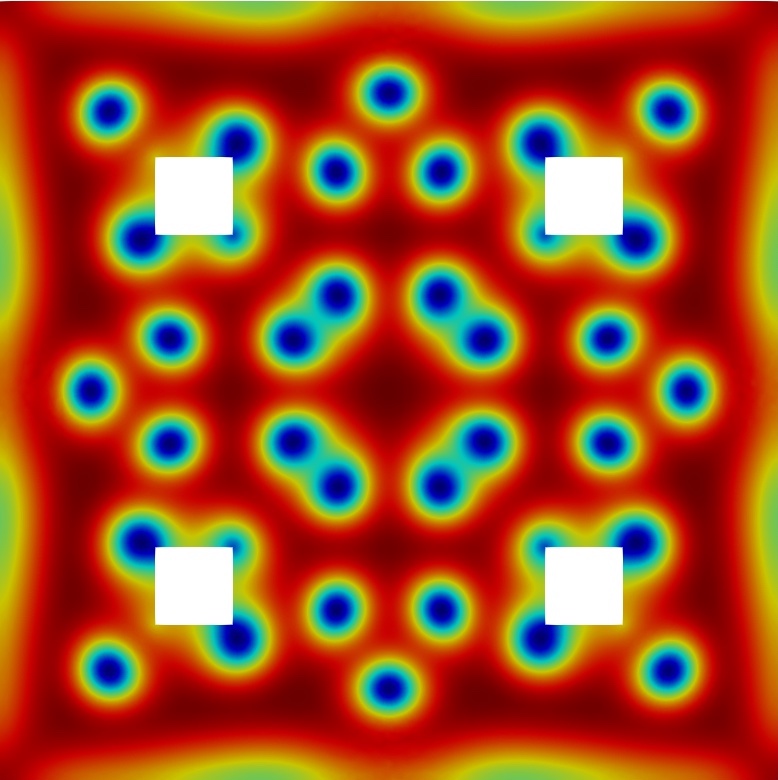}
	\includegraphics[width = 1.2in]{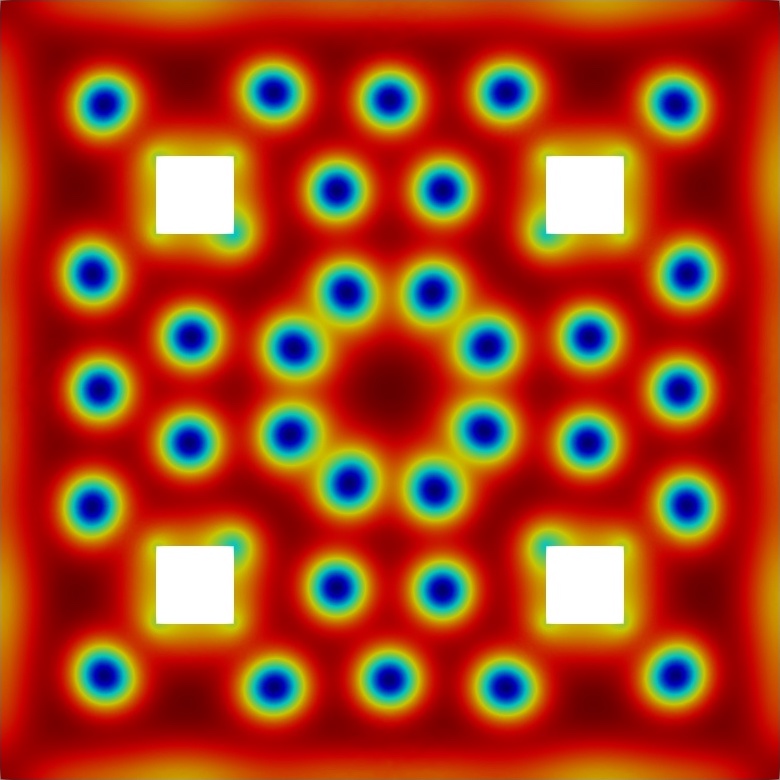}
	\includegraphics[width = 1.2in]{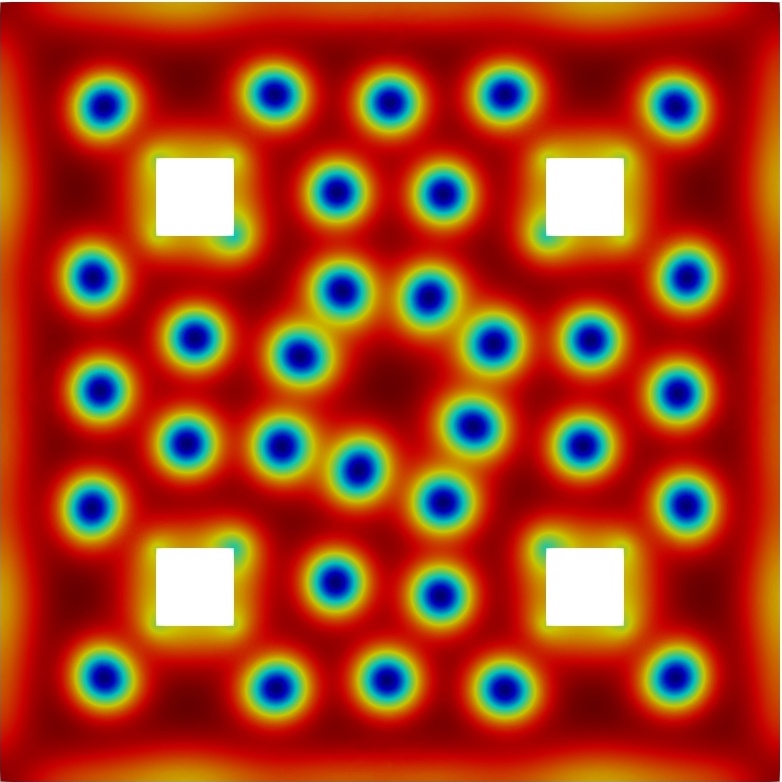} 
	\includegraphics[width = 1.2in]{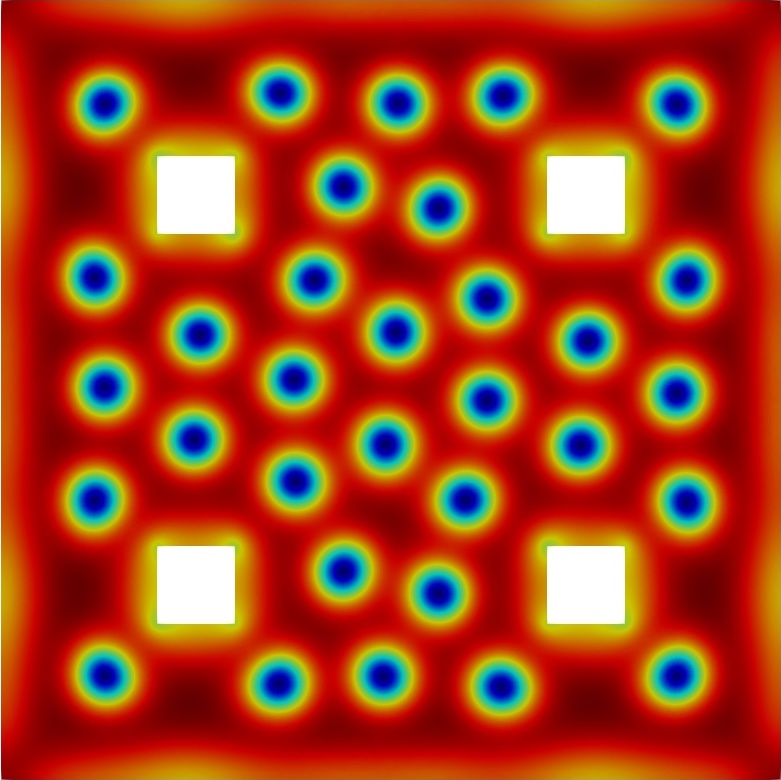} 
}

\subfloat[$\nabla\times \bA$  at $T=10$, $20$, $50$, $100$ and  $500$ ]{%
\centering
	\includegraphics[width = 1.2in]{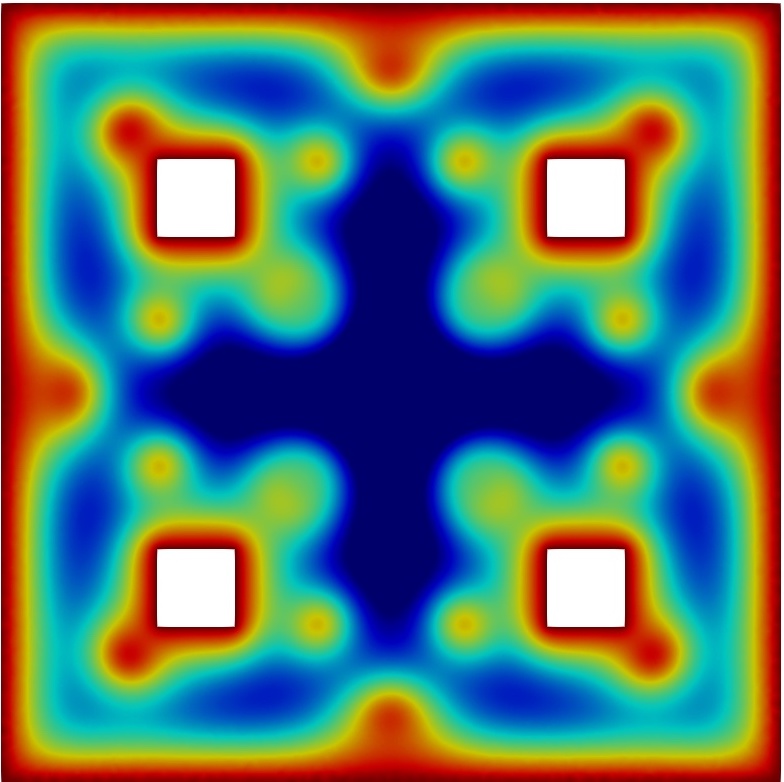}
	\includegraphics[width = 1.2in]{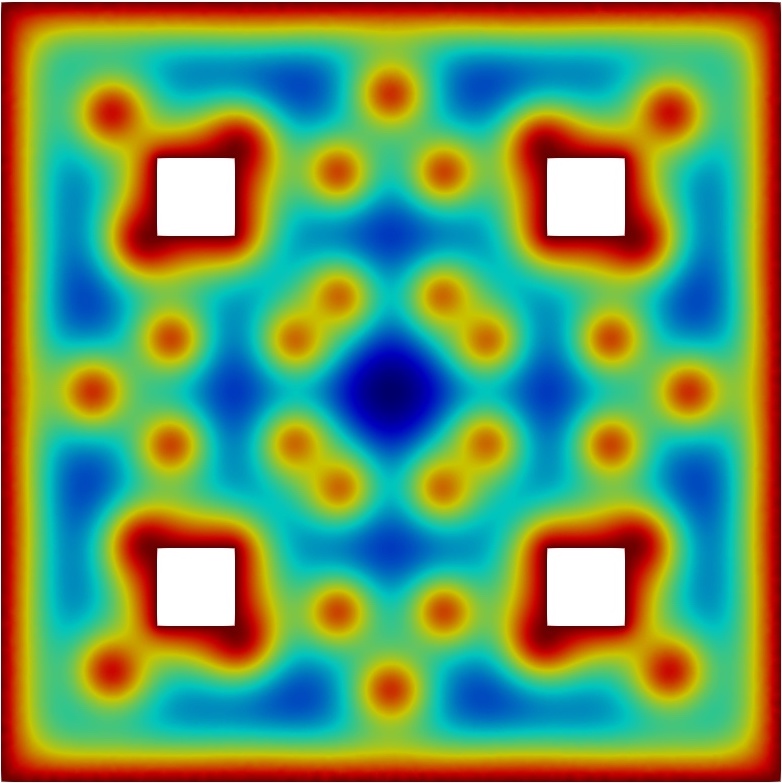}
	\includegraphics[width = 1.2in]{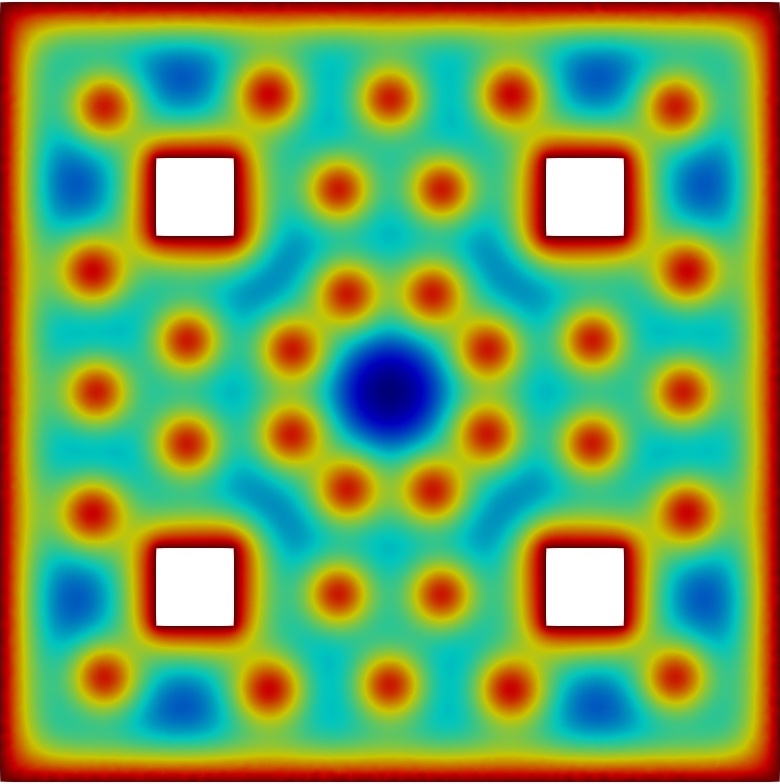}
	\includegraphics[width = 1.2in]{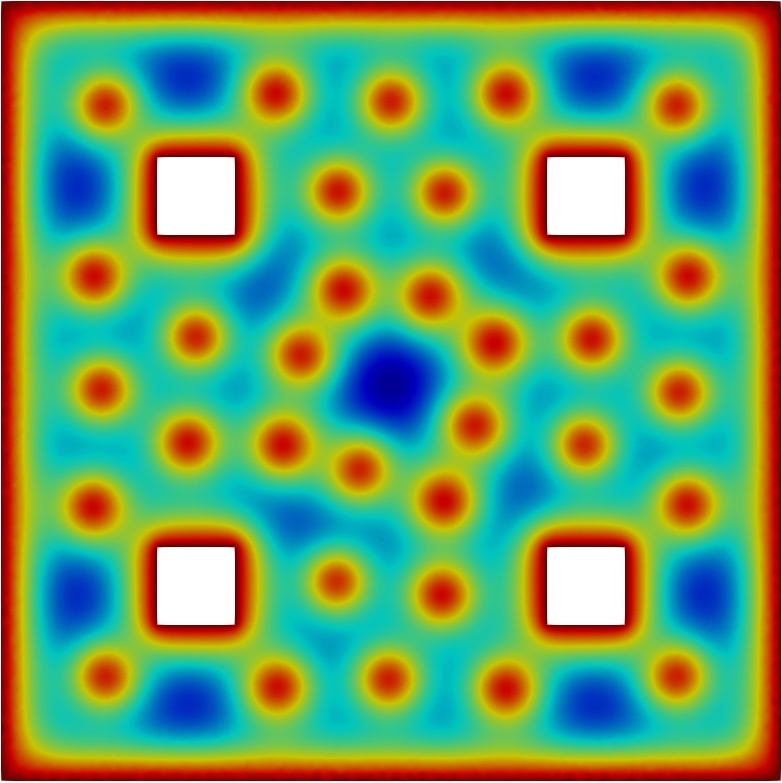} 
	\includegraphics[width = 1.2in]{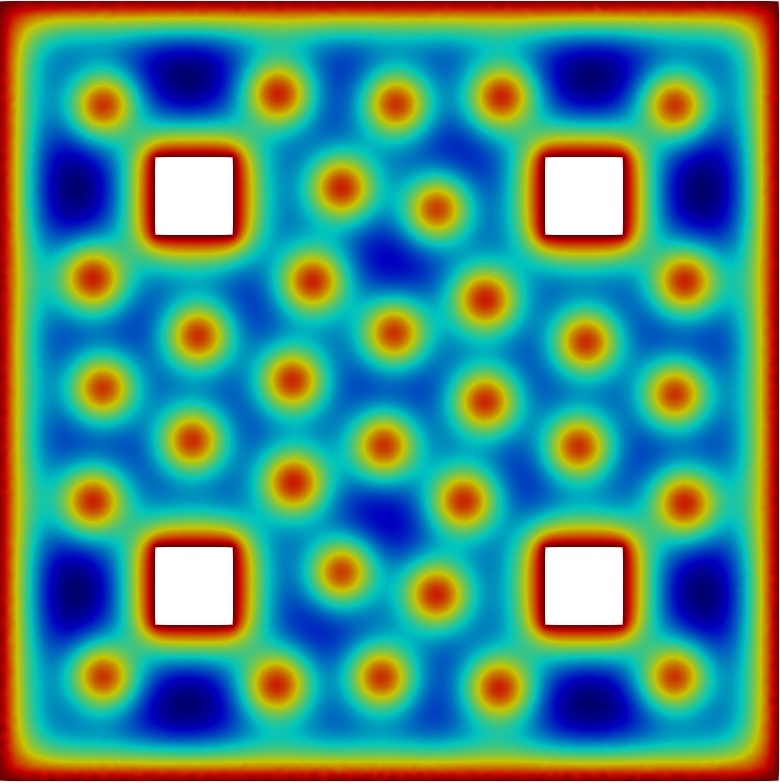} 
}

\caption{\footnotesize $|\psi|^2$ and $\nabla\times \bA$ on the squared domain with four holes for $\bH=1.1$.}
\label{fig:hole11}
\end{figure}

\section{Conclusions}
A new nonlinear finite element approach is proposed for solving the time dependent Ginzburg-Landau equations  under the temporal gauge with the original boundary condition. This numerical scheme solves the magnetic  potential by the lowest order of the  second kind ${\rm N\acute{e}d\acute{e}lec}$ element. This offers the advantage to deal with the original boundary condition of the physical problem directly, instead of requiring some additional boundary conditions to guarantee the wellposedness of the discrete system. The conventional finite element scheme solves the magnetic potential in a relatively smaller space with higher regularity. Compared to this conventional method, the proposed approach is more stable and reliable when dealing the superconductor with reentrant corners as showed in the numerical tests. The wellposedness and energy stable property of the nonlinear scheme is analyzed under some condition. 
The Newton method is applied to solve the proposed nonlinear system, and two efficient preconditioners are designed to speed up the simulations. The boundedness and the coercivity of the bilinear forms with respect to the proposed preconditioners are analyzed under some conditions. This motivates the design of the preconditioners. This efficient preconditioner plays an  important role in speeding up the simulation and makes the computational cost of this nonlinear system comparable to that of a linear system.
The comparison in numerical simulations verifies the efficiency of the proposed preconditioner.

\bibliographystyle{plain}
\bibliography{bibsuperConduct}
\end{document}